\documentclass[reqno,11pt]{article}
\usepackage{mathrsfs}
\usepackage{amssymb}
\usepackage{amsthm}
\usepackage{amsmath}
\usepackage{amsfonts}
\usepackage{mathtools}
\usepackage{enumerate}
\usepackage{bm}

\usepackage{graphicx}

\usepackage{bbm}
\usepackage{mathrsfs}
\usepackage{amssymb}
\usepackage[thicklines]{cancel}

\UseRawInputEncoding

\usepackage[usenames,dvipsnames]{xcolor}
\usepackage{soul}

\usepackage{txfonts}

\usepackage{psfrag}

\usepackage{setspace,color,wrapfig}
\usepackage[colorlinks,linkcolor=blue,citecolor=cyan]{hyperref}

\numberwithin{equation}{section}
\newtheorem{theorem}{Theorem}[section]
\newtheorem{lemma}[theorem]{Lemma}

\newtheorem{proposition}[theorem]{Proposition}
\newtheorem{remark}[theorem]{Remark}
\theoremstyle{definition}

\newtheorem{problem}[theorem]{Problem}
\theoremstyle{remark}

\begin{document}
\title{  Transonic shocks  for  three-dimensional axisymmetric    Euler flows  with  an  external force   in a perturbed  cylinder}
\author{
Zihao Zhang\thanks{School of Mathematics, Jilin University, Changchun, Jilin Province, China, 130012. Email: zhangzihao@jlu.edu.cn}}
\date{}

\newcommand{\de}{{\mathrm{d}}}
\def\div{{\rm div\,}}
\def\curl{{\rm curl\,}}
\newcommand{\ro}{{\rm rot}}
\newcommand{\sr}{{\rm supp}}
\newcommand{\sa}{{\rm sup}}
\newcommand{\va}{{\varphi}}
\newcommand{\me}{\mathcal{M}}
\newcommand{\ml}{\mathcal{V}}
\newcommand{\mi}{\mathcal{N}}
\newcommand{\md}{\mathcal{D}}
\newcommand{\mg}{\mathcal{G}}
\newcommand{\mh}{\mathcal{H}}
\newcommand{\mf}{\mathcal{F}}
\newcommand{\ms}{\mathcal{S}}
\newcommand{\mt}{\mathcal{T}}
\newcommand{\mn}{\mathcal{N}}
\newcommand{\ma}{\mathcal{L}}
\newcommand{\mb}{\mathcal{P}}
\newcommand{\mj}{\mathcal{J}}
\newcommand{\n}{\nabla}
\newcommand{\h}{\hat}
 \newcommand{\q}{{\rm R}}
\newcommand{\p}{{\partial}}
\renewcommand\figurename{\scriptsize Fig}
\pagestyle{myheadings} \markboth{  Three-dimensional axisymmetric   transonic shock  flows  with  an  external force}{Three-dimensional axisymmetric   transonic shock  flows with  an  external force}\maketitle
\begin{abstract}
   We establish the existence and stability of the  transonic shock solution to  three-dimensional  axisymmetric Euler
system with an external force in a cylinder   under   perturbations of the incoming supersonic flow,  the exit pressure, the external force and the nozzle wall.
  The external force has  a
stabilization effect on the transonic shock in the  straight cylinder and the shock position is uniquely determined.   The main difficulties
for the axisymmetric flows  are  the corner singularities near the intersection point of the shock surface and the nozzle wall and the singularity near the symmetry axis.  An  invertible modified Lagrangian transformation  is  introduced to  overcome these difficulties and   straighten
the streamline.  One of the key elements in the analysis is to
   decompose the hyperbolic and elliptic modes for the steady axisymmetric Euler system with an external force in terms of the deformation and vorticity.   Another one is an equivalent reformulation of the Rankine-Hugoniot conditions so that the
shock front is uniquely determined by an algebraic equation.
\end{abstract}
\begin{center}
\begin{minipage}{5.5in}
\end{minipage}
\end{center}
\section{Introduction and main results}\noindent
\par In this paper, we  study  the  transonic shock problem for three-dimensional axisymmetric    Euler flows
of isentropic   polytropic gases in a  perturbed cylinder under the external force. Assume the flow enters the nozzle with a supersonic state and leaves it with a relatively high pressure, then it is expected that a shock front occurs in the nozzle such that the flow pressure rises to coincide with the pressure at the exit. Then catching the position of the shock front is one of the  important ingredients in determining the flow field in the nozzle. This paper shows  that the external force has  a
stabilization effect on the transonic shocks  in  the straight cylinder  and the shock position is uniquely determined. Then we further investigate the structural
stability of   transonic shock solutions under axisymmetric perturbations of the incoming supersonic flow,  the exit pressure, the external force and the nozzle wall.

\par  Three-dimensional steady Euler flows  with an external force  is governed by the
 following system:
  \begin{equation}\label{1-1}
\begin{cases}
\p_{x_1}(\rho u_1)+\p_{x_2}(\rho u_2)+\p_{x_3}(\rho u_3)=0,\\
\p_{x_1}(\rho u_1^2+P)+\p_{x_2}(\rho u_1u_2)+\p_{x_3}(\rho u_1u_3)=\rho \p_{x_1}\Phi,\\
\p_{x_1}(\rho u_1u_2)+\p_{x_2}(\rho u_2^2+P)+\p_{x_3}(\rho u_2u_3)=\rho \p_{x_2}\Phi,\\
\p_{x_1}(\rho u_1u_3)+\p_{x_2}(\rho u_2u_3)+\p_{x_3}(\rho u_3^2+P)=\rho \p_{x_3}\Phi,\\
\end{cases}
\end{equation}
where $ \textbf{u} = (u_1, u_2,u_3) $ is the velocity field, $ \rho $ is the density, $ P $ is the pressure and  $\Phi $ is the potential force, respectively.  We  consider the  isentropic  polytropic gases, therefore the equation of state is given by $ P=A\rho^{\gamma} $, where  $ A $ is a positive constant and $ \gamma $ is the adiabatic constant with $ \gamma> 1 $. For convenience, we take $A=1$ in
this paper. Denote the  sound speed by  $ c(\rho)=\sqrt{P^{\prime}(\rho)}$. It is well-known that  system \eqref{1-1} is  hyperbolic  for supersonic flows (i.e. $ |\textbf{u}|>c(\rho) $) and  hyperbolic-elliptic mixed  for subsonic flows (i.e. $ |\textbf{u}|<c(\rho) $).
\par The stability analysis of  transonic shock solutions  in a flat nozzle have been studied extensively. For steady flows with shocks in  finitely and infinitely long  flat nozzles, there exists a class of transonic shock solutions with both upstream supersonic state and downstream subsonic state being constant and its shock position being arbitrary. The structural stability of these transonic shocks for steady potential flows in nozzles was studied in \cite{CGF03,CFM04,CQM07,XY05, YX08}. The authors in  \cite{CCS06,CCM07} established the existence of transonic shocks to steady Euler flows in two-dimensional nozzles with slowly varying
cross-sections.  The existence and stability of the transonic shock for  two-dimensional and three-dimensional steady Euler flows in flat or almost flat nozzles  with the prescribed pressure at the exit up to a constant were studied in \cite{CS05,XYY09} and
 \cite{CS08,CY08}.
 Both existence results are established under the assumption
that the shock front passes through a given point. Recently,  without such an artificial assumption, the authors in \cite{FX21}  established the stability and existence of transonic shock solutions to the
two-dimensional steady compressible Euler system in an almost flat finite nozzle with the exit pressure, where the shock position was uniquely determined. This was generalized to three-dimensional axisymmetric case in \cite{FG21}.
 \par On the other hand, there were many studies on the stability of the radially symmetric transonic shock in a divergent nozzle. 
The well-posedness of the transonic shock problem in two-dimensional divergent nozzles under the perturbations of the exit pressure was first established    in \cite{LXY09}  when the opening
angle of the nozzle is suitably small. This restriction was removed in  \cite{LYX09} and the transonic shock in
a two-dimensional straight divergent nozzle is shown in \cite{LXY13} to be structurally stable under the
perturbations of the nozzle walls and the exit pressure. The existence and stability of
three-dimensional axisymmetric transonic shock flows  in a conic nozzle were studied in \cite{LYX10,LXY10,WXX21,ZZ22}. In \cite{WXX21}, the authors  introduced a
modified Lagrangian transformation     to deal
with the corner singularities near the intersection points of the shock surface and nozzle boundary
and the artificial singularity near the axis simultaneously.  The stability of spherically symmetric transonic shocks in a spherical shell was studied in \cite{LXY16} by requiring that the background transonic shock solutions satisfy some ``Structure Conditions". Recently, the authors in \cite{WX23} had made a substantial progress and established the existence and stability of cylindrical transonic shock solutions under three-dimensional perturbations of the incoming flows and the exit pressure without any restriction on the background transonic shock solutions.
\par Let $ L_1, L_2(>L_1) $ be fixed positive constants.  The straight cylinder (Fig 1)  is given  by
\begin{equation*}
 \mi_b:=\{(x_1,x_2,x_3)\in \mathbb{R}^3: L_1<x_1<L_2, 0\leq x_2^2+x_3^2<1\}.
 \end{equation*}
 \begin{figure}
  \centering
  \includegraphics[width=9cm,height=4cm]{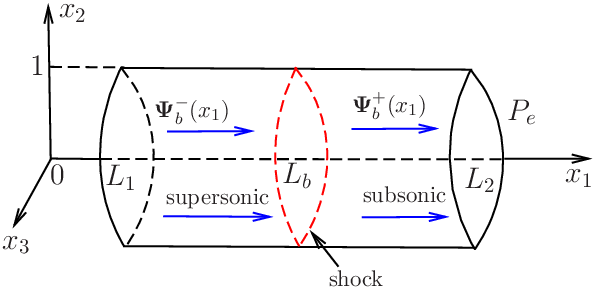}
  \caption{ The transonic shock flows in the straight cylinder}
\end{figure}
We first consider  one-dimensional steady Euler system with an external force in  $\mi_b$, which is governed by
 \begin{equation}\label{1-2}
 \begin{cases}
 (\rho_b u_b)^{\prime}(x_1)=0, \\
  (\rho_b u_bu_b^{\prime})(x_1)+P_b^{\prime}(x_1)=(\rho_b g)(x_1),\\
  {\rho_b}(L_1)=\bar\rho>0,\ \ {u_b}(L_1)= \bar u>0,\\
        {P_b} (L_2) = P_e,
\end{cases}
  \end{equation}
 where  the flow state at the entrance $x_1=L_1$ is supersonic, i.e., $\bar u^2>c^2(\bar\rho)=\gamma
\bar\rho^{\gamma-1}$. By employing the monotonicity relation between the shock position and the end pressure,    the following Lemma was established   in \cite{WY22}  shows that    there is a unique transonic shock solution to \eqref{1-2} when the end pressure is a suitably prescribed constant $P_e$ and $g (x_1) > 0$ for any $ x_1 \in [L_1, L_2]$. Meanwhile, it is shown that the external force has  a
stabilization effect on the transonic shock in the cylinder and the shock position is uniquely determined.
\begin{lemma}
  Suppose that the initial state $(\bar \rho,\bar u)$
at $x_1=L_1$ is supersonic and the external force $g$ satisfying $ g(x_1)>0 $ for any $x_1 \in [L_1, L_2]$, there exist two positive constants $P_1,P_2$ such that if the end pressure $P_e\in (P_1, P_2)$, there exists a unique piecewise transonic shock solution
\begin{equation}\label{1-3}
\mathbf{\Psi}_b(\mathbf{x})=(\mathbf{u}_b,P_b)(\mathbf{x})=
\begin{cases}
\mathbf{\Psi}_b^-(\mathbf{x}):=(u_b^-(x_1),0,0,P_b^{-}(x_1)),\ \ {\rm{if}}\ L_1<x_1<L_b,\\
\mathbf{\Psi}_b^+(\mathbf{x}):=(u_b^+(x_1),0,0,P_b^{+}(x_1)),\ \ {\rm{if}}\ L_b<x_1<L_2,\\
\end{cases}
\end{equation}
with a shock front located at $ x_1=L_b\in(L_1,L_2) $. Across the shock, the following Rankine-Hugoniot conditions and entropy condition are satisfied:
 \begin{equation*}
 \begin{cases}
  [{ \rho_b} { u}_b](L_b)=0,\\
	[{\rho}_b u_b^2+P_b](L_b)=0,\\
[P_b](L_b)>0.
\end{cases}
 \end{equation*}
 Moreover, the shock position $ x_1=L_b $ increases as the exit pressure $ P_e $ decreases. In addition, the shock position $ x_1=L_b $ approaches to $ L_1 $ if $ P_e $ goes to $ P_2 $ and $ x_1=L_b $ approaches to $ L_2 $ if $ P_e $ goes to $ P_1 $.
\end{lemma}
\par The one-dimensional transonic shock solution $ \mathbf{\Psi}_b$ with a shock occurring at $ x_1=L_b $  will be called the background solution in this paper.  Clearly, one can extend the supersonic and subsonic parts
of $\mathbf{\Psi}_b $ in a natural way, respectively. For convenience, we still call the
extended subsonic and supersonic solutions $  \mathbf{\Psi}_b^+$
and $  \mathbf{\Psi}_b^- $. This paper is going to establish the structural
stability of this transonic shock solution under axisymmetric
perturbations of the incoming supersonic flows,   the exit pressure, the external force and the nozzle wall.
\par
 Introduce the cylindrical coordinates $(x, r, \theta)$:
\begin{equation*}
x=x_1,\ r=\sqrt {x_2^2 + x_3^2},\ \theta  = \arctan \frac{{{x_3}}}{{{x_2}}}.
\end{equation*}
Any function $h({\bf x})$ can be represented as $h({\bf x})=h(x,r,\theta)$, and a vector-valued function ${\bf H}({\bf x})$ can be represented as ${\bf H}({\bf x})=H_x(x,r,\theta){\bf e}_x+ H_r(x,r,\theta){\bf e}_r+ H_{\theta}(x,r,\theta){\bf e}_{\theta}$,
where
\begin{equation*}
{\bf e}_x=(1,0,0)^T,\quad {\bf e}_r=(0,\cos\theta, \sin\theta)^T,\quad {\bf e}_{\theta}=(0,-\sin\theta,\cos\theta)^T.
\end{equation*}
We say that a function $h({\bf x})$ is  axisymmetric if its value is independent of $\theta$ and that a vector-valued function ${\bf H }= (H_x, H_r, H_{\theta})$ is axisymmetric if each of functions $H_x({\bf x}), H_{r}({\bf x})$ and $H_{\theta}({\bf x})$ is axisymmetric.
\par Assume that the density, the pressure and the velocity are of the form
\begin{equation*}
\begin{aligned}
\rho({\textbf x})=\rho(x,r),\quad P({\textbf x})=P(x,r), \quad
{\bf u}({\textbf x})= u_x(x,r) {\bf e}_{x}+ u_r(x,r) {\bf e}_{r}+ u_{\theta}(x,r) {\bf e}_{\theta}.
\end{aligned}
\end{equation*}
Then \eqref{1-1} can be simplified as
 \begin{equation}\label{1-4}
\begin{cases}
\begin{aligned}
&\p_{x}(r\rho u_x)+\p_{r}(r\rho u_r)=0,\\
&\rho(u_x\p_{x}+u_r\p_{r})u_x+\p_{x}P=\rho\p_{x}\Phi,\\
&\rho(u_x\p_{x}+u_r\p_{r})u_r-\frac{\rho u_{\theta}^2}{r}+\p_{r}P=\rho\p_{r}\Phi,\\
&\rho(u_x\p_{x}+u_r\p_{r})(ru_{\theta})=0.
\end{aligned}
\end{cases}
\end{equation}
We  perturb the cylinder (Fig 2) as
\begin{equation*}
\mi:=\{(x,r)\in \mathbb{R}^2: L_1<x<L_2,\ 0\leq r<1+\sigma f(x)\},
\end{equation*}
where $ \sigma $ is sufficiently small and $ f\in C^{2,\alpha}([L_1,L_2])$ satisfies
\begin{equation}\label{1-5}
 f(L_1)=f^{\prime}(L_1)=0 .
 \end{equation}
 
 \par  Let the potential force $ \Phi $ and the  incoming supersonic flow at the inlet $ x=L_1 $ be prescribed as
\begin{equation}\label{1-6}
\begin{cases}
\Phi(x,r)= \Phi_b(x)+\sigma \Phi_e(x,r),\\
\mathbf{\Psi}^-(L_1,r)=\mathbf{\Psi}_b^-(L_1)
+\sigma(u_{en}^-, v_{en}^-,w_{en}^-, P_{en}^-)(r).
\end{cases}
\end{equation}
Here $ \Phi_b^{\prime}=g $, $ \Phi_e \in C^{2,\alpha}(\overline{ \mn}) $ and
$(u_{en}^-, v_{en}^-, w_{en}^-, P_{en}^-)\in (C^{2,\alpha}[0,1])^4 $.  On the nozzle wall, the flow satisfies the slip condition $  \textbf{u}\cdot \textbf{n}=0 $, where $ \textbf{n} $ is the outer normal of the nozzle wall. Using cylindrical coordinates, the slip boundary condition  can be rewritten as
\begin{equation}\label{1-7}
u_r=\sigma f^{\prime}(x)u_x, \ \ {\rm{on}}\ \ \Gamma=\{(x,r):r= 1+\sigma f(x),\ L_1\leq x \leq L_2\}.
\end{equation}
On the exit of the nozzle, the end pressure is prescribed by
\begin{equation}\label{1-8}
P(L_2,r)=P_e+\sigma P_{ex}(r),
\end{equation}
where  $ P_{ex}(r)\in C^{1,\alpha}([0,1+\sigma f(L_2)]) $. 
\begin{figure}
  \centering
  \includegraphics[width=9cm,height=4cm]{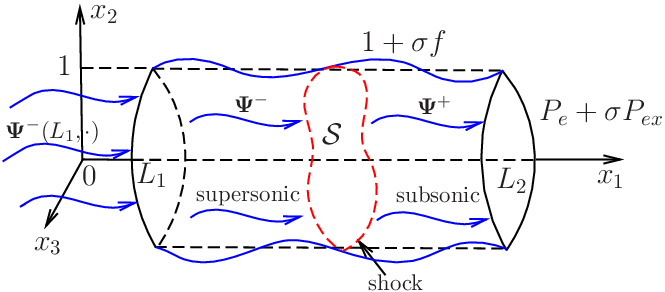}
  \caption{ The stability problem }
\end{figure}
  \par In this paper, we want to look for a piecewise smooth solution $ \mathbf{\Psi} $, which jumps only at a shock front $ \ms=\{(x,r):x=\xi(r), r\in[0,r_\ast]\} $. Here $r_\ast $ stands for  the intersection circle of the shock surface with the nozzle wall. More precisely, $ \mathbf{\Psi} $ has the following form
\begin{equation}\label{1-10}
\mathbf{\Psi}=
\begin{cases}
\mathbf{\Psi}^-:=(u_x^-,u_r^-,u_{\theta}^-,P^-)(x,r), \ \ {\rm{in}}\ \mn_-=\{L_1<x<\xi(r), \ 0\leq r< 1+\sigma f(x)\},\\
\mathbf{\Psi}^+:=(u_x^+,u_r^+,u_{\theta}^+,P^+)(x,r), \ \ {\rm{in}}\ \mn_+=\{\xi(r)<x<L_2,\ 0\leq r< 1+\sigma f(x)\},\\
\end{cases}
\end{equation}
  and satisfies the following  Rankine-Hugoniot conditions on  the shock surface $ \ms $:
\begin{equation}\label{1-11}
\begin{cases}
[\rho u_x]-\xi^{\prime}(r)[\rho u_r]=0,\\
[\rho u_x^2+P]-\xi^{\prime}(r)[\rho u_xu_r]=0,\\
[\rho u_xu_r]-\xi^{\prime}(r)[\rho u_r^2+P]=0,\\
[\rho u_xu_{\theta}]-\xi^{\prime}(r)[\rho u_ru_{\theta}]=0.\\
\end{cases}
\end{equation}
\par The existence and uniqueness of the supersonic flow to \eqref{1-4} follows from the classical theory  to the boundary value problem for quasi-linear hyperbolic systems (See \cite{LY85}).
\begin{lemma}
Assume that the potential
 force and  the  incoming supersonic data given in \eqref{1-6} satisfying the following compatibility conditions:
\begin{equation}\label{1-13}
\begin{cases}
\begin{aligned}
&\p_r\Phi_e(x,0)=0,\\
&v_{en}^-(0)=w_{en}^-(0)=(v_{en}^-)''(0)=(w_{en}^-)^\prime(0)=
(P_{en}^-)^\prime(0)=0,\\
&v_{en}^-(1)=0,\ (P_{en}^-)^\prime(1)=\rho_{en}^-(1)((w_{en}^-)^2(1)+\sigma\p_r\Phi_e(L_1,1)).\\
\end{aligned}
\end{cases}
\end{equation}
Then there exists a constant $\sigma_0>0$ depending only on the background solution and the boundary data, such that for any $0<\sigma<\sigma_0$, there exists a unique  axisymmetric solution $ \mathbf{\Psi}^-=(u_x^-,u_r^-,u_{\theta}^-,P^-)
\in C^{2,\alpha}(\overline{\mi}) $ to \eqref{1-4}  with \eqref{1-6}-\eqref{1-7}, which satisfies
\begin{equation}\label{1-14}
  \|(u_x^-,u_r^-,u_{\theta}^-,P^-)-(u_b^-,0,0,P_b^{-})\|_{C^{2,\alpha}(\overline{\mi})}\leq C_0 \sigma,
  \end{equation}
  and
  \begin{equation}\label{1-15}
 (u_r^-, {\p_r^2} u_{r}^-)(x,0)=(u_{\theta}^-,\p_ru_{\theta}^-)(x,0)=\p_r(u_x^-,P^-)(x,0)
=0, \quad \forall x\in[L_1,L_2].
  \end{equation}
  \end{lemma}
  Therefore, the problem is reduced to solve a free boundary value problem for the steady axisymmetric  Euler system with an external force in which the shock front and the downstream subsonic flows are unknown. To state the main result,   we first introduce some weighted H\"{o}lder spaces and their norms.  For  any bounded domain  $ \md\in\mathbb{R}^3 $,  let $ \mh  $ be a closed portion of $ \md $.  For $ \mathbf{x},\tilde{\mathbf{x}}\in \md $, define
  \begin{equation*}
 \delta_{\mathbf{x}}:=\rm{dist}(\mathbf{x},\mh)\quad {\rm{and}} \quad
  \delta_{\textbf{x},\tilde{\textbf{x}}}:
  =\min(\delta_{\mathbf{x}},\delta_{\tilde{\mathbf{x}}}).
  \end{equation*}
  For any positive integer $ m $, $ \alpha\in(0,1) $ and $ \kappa\in \mathbb{R} $, we define
  \begin{equation*}
  \begin{aligned}
  {\|u\|}_{m,0;\md}^{(\kappa;\mh)}\ &:=\sum_{0\leq|\beta|=m}\sup_{\textbf{x}\in \md} \delta_{\mathbf{x}}^{\max(|\beta|+\kappa,0)}|D^{\beta}u(\textbf{x})|,;\\
  [u]_{m,\alpha;\md}^{(\kappa;\mh)}\ &:=\sum_{|\beta|=m}\sup_{ {\textbf{x}},\tilde {\textbf{x}}\in \md,{\textbf{x}}\neq \tilde {\textbf{x}}}\delta_{\mathbf{x},\tilde{\mathbf{x}}}^{\max(m+\alpha+\kappa,0)}
  \frac{|D^{\beta}u(\textbf{x})-D^{\beta}u(\tilde{\textbf{x}})|}{|\textbf{x}-\tilde{\textbf{x}} |^{\alpha}};\\
  \end{aligned}
  \end{equation*}
  \begin{equation*}
  \begin{aligned}
  \|u\|_{m,\alpha;\md}^{(\kappa;\mh)}\ &:= {\|u\|}_{m,0;\md}^{(\kappa;\mh)}
  +[u]_{m,\alpha;\md}^{\kappa;\mh}.\\
  \end{aligned}
  \end{equation*}
  where $D^{\beta}$ denotes $\p_{x_1}^{\beta_1}\p_{x_2}^{\beta_2}\p_{x_3}^{\beta_3} $ for a multi-index $\beta= (\beta_1,\beta_2,\beta_3)$ with $\beta_j\in \mathbb{Z}_+$ and $|\beta|=\sum_{j=1}^3\beta_j$.
  $ C_{m,\alpha}^{(\kappa;\mh)}(\md) $ denotes the completion of the set of all smooth functions whose $\|\cdot\|_{m,\alpha;\md}^{(\kappa,\mh)}$ norms are finite.
  Furthermore, for a vector function $ {\bf u}=(u_1,u_2,\cdots,u_n) $, define
  \begin{equation*}
   \|{\bf u}\|_{m,\alpha;\md}^{(\kappa,\mh)}:=\sum_{i=1}^{n}\| u_i\|_{m,\alpha;\md}^{(\kappa,\mh)}.
\end{equation*}
  \par The main result in this paper
is stated as follows.
\begin{theorem}
  Assume that the compatibility conditions \eqref{1-5} and \eqref{1-13} hold. There exist suitable positive constants $ \sigma_0 $ and $ C_\ast$ depending only on the background solution $ \mathbf{\Psi}_b $ defined in \eqref{1-3} and the boundary data $ \mathbf{\Psi}^-(L_1,\cdot) $, $ f $, $ P_{ex} $ ,$ \Phi_e $ such that if $ 0< \sigma \leq \sigma_0 $, the problem \eqref{1-4} with \eqref{1-6}, \eqref{1-7},   \eqref{1-8} and \eqref{1-11}  has a unique axisymmetric solution $  \mathbf{\Psi}^+=(u_x^+,u_r^+,u_{\theta}^+,P^+) $ with the shock front $ \ms $ satisfying the following properties.
\begin{item}[\rm{(1)}]
The function $ \xi(r) \in C_{3,\alpha}^{(-1-\alpha;\{r_\ast\})}([0,r_\ast)) $ satisfies
\begin{equation}\label{1-17}
\|\xi(r)-L_b\|_{3,\alpha;[0,r_\ast)}^{(-1-\alpha;\{r_\ast\})} \leq C_\ast\sigma,
\end{equation}
and
\begin{equation}\label{1-18}
\xi^\prime(0)=\xi^{(3)}(0)=0.
\end{equation}
 \end{item}
\begin{item}[\rm{(2)}]
The solution $ \mathbf{\Psi}^+=(u_x^+,u_r^+,u_{\theta}^+,P^+)\in C_{2,\alpha}^{(-\alpha;\Gamma_{p,s})}(\mi_+) $ satisfies   the entropy condition
 \begin{equation}\label{1-12}
P^{+}(\xi(r),r)>P^{-}(\xi(r),r), \ \forall r\in [0,r_\ast],
\end{equation}
and the estimate
\begin{equation}\label{1-19}
\|\mathbf{\Psi}^+-\mathbf{\Psi}_b^+\|_{2,\alpha;\mi_+}^{(-\alpha;\Gamma_{p,s})}\leq C_\ast\sigma
\end{equation}
 with the compatibility conditions
 \begin{equation}\label{1-20}
(u_r^+, {\p_r^2} u_{r}^+)(x,0)=(u_{\theta}^+,\p_ru_{\theta}^+)(x,0)=\p_r(u_x^+,P^+)(x,0)
=0, \quad \forall x\in[\xi(r),L_2].
  \end{equation}
Here
\begin{equation*}
\Gamma_{p,s}=\{(x,r):\xi(r) \leq x\leq L_2, r=1+\sigma f(x)\}.
\end{equation*}
\end{item}
\end{theorem}
\begin{remark}
Compared with the two-dimensional case in \cite{WY22}, one of the major difficulties for
axisymmetric flows is the possible singularity near the symmetric axis. Inspired by \cite{WXX21},  the singular term $ r $ in the density equation is of order $ O(r) $ near the axis $ r=0 $, hence we can  find a simple  modified
Lagrangian transformation  such that it is invertible near the axis and also straightens
the streamline.
\end{remark}
\begin{remark}
 For generic perturbations of the nozzle wall,  the corner singularity will be transported along the trajectory and the velocity field will develop singularity and can not be Lipschitz near the cylinder wall in general. Thus one can only obtain  the optimal $C^{\alpha}(\overline{\mn_+})$ regularity  in Theorem 1.3. This fact
has been shown in Lemma 3.3 and Remark 3.2 of  \cite{XYY09}.  If the nozzle wall is not perturbed,  we can  show the higher order regularity for both the flows and the shock surface under the higher compatibility conditions of the exit pressure.
\end{remark}
\begin{remark}
 The previous work \cite{FG21,LYX10,WXX21} reduced the steady axisymmetric Euler system in the subsonic region  into a  elliptic system of the flow angle and
pressure. One of main ingredients of our analysis here is quite
different from those in \cite{FG21,LYX10,WXX21},
   we   decouple the hyperbolic and elliptic modes in terms of the deformation and vorticity developed in \cite{WX19,WS19},  and reduce the steady axisymmetric Euler system with an external force to two transport equations for the Bernoulli's quantity and the swirl velocity  and a deformation-curl system for the velocity field. Another one is the reformulation of the Rankine-Hugoniot conditions,
which determines the shock front by an algebraic equation and also yields suitable boundary conditions for the hyperbolic
and elliptic modes.
\end{remark}
\par The rest of this article is organized as follows.   In Section 2, we  introduce the modified Lagrangian transformation  and decompose  the hyperbolic and elliptic modes for the steady axisymmetric   Euler system with an external force in the subsonic region in terms of the deformation and curl, and the corresponding reformulation of the Rankine-Hugoniot  conditions.  We also introduce a coordinate transformation such that  the free boundary becomes fixed.
In Section 3,   we design an iteration scheme to prove  Theorem 1.3.
\section{The reformulation of the transonic shock problem}\noindent
\par In this  section, we first introduce the modified Lagrangian transformation to straighten the streamline, then  the deformation-curl reformulation  developed in \cite{WX19,WS19} is employed to  rewrite the steady axisymmetric Euler system with an external force. Finally, we reformulate the Rankine-Hugoniot conditions and boundary conditions  and introduce another coordinates transformation   to reduce the transonic shock problem into a fixed boundary value problem.
\subsection{The modified Lagrangian transformation}\noindent
\par  As we mentioned in Remark 1.5,  one
can only expect the $C^{\alpha} $ boundary regularity for the solution in the subsonic region.  In order to avoid the difficulty in uniquely determining the trajectory, we introduce a Lagrangian transformation to straighten the streamline. However, in the three-dimensional axisymmetric setting,  there is a singular term $ r $ in the density equation.  Inspired by \cite{WXX21}, we can  find a simple  modified
Lagrangian transformation   to overcome this difficulty and apply this modified Lagrangian transformation to rewrite \eqref{1-4} and \eqref{1-11}.
\par Define $  (\tilde y_1,\tilde y_2)=(x,\tilde y_2(x,r)) $ such that
\begin{equation}\label{2-1}
\begin{cases}
\frac{\p \tilde y_2}{\p x}=-r\rho^-u_r^-, \
\frac{\p \tilde y_2}{\p r}=r\rho^-u_x^-, \quad {\rm{if}}\
(x,r)\in \overline{\mi_-},\\
\frac{\p \tilde y_2}{\p x}=-r\rho^+u_r^+, \
\frac{\p \tilde y_2}{\p r}=r\rho^+u_x^+, \quad {\rm{if}}\
(x,r)\in \overline{\mi_+},\\
\tilde y_2(L_1,0)=0,\quad \tilde y_2(L_2,0)=0.\\
\end{cases}
\end{equation}
On the axis $ r=0 $ and the nozzle wall $ \Gamma $, one  derives that
\begin{equation*}
\frac{\de} {\de x}\tilde y_2(x,0)=0 \quad {\rm{and}}\quad
\frac{\de} {\de x}\tilde y_2(x,1+\sigma f(x))=0.
\end{equation*}
Thus for  any $  x\in[L_1,L_2] $, we can assume $ \tilde y_2(x,0)=0 $. Then one has
\begin{equation*}
\begin{cases}
\tilde y_2(x,1+\sigma f(x))=\me^2, \quad  \forall x\in[L_1,L_\ast],\\
\tilde y_2(x,1+\sigma f(x))=\me_1^2, \quad \forall x\in[L_\ast,L_2],
\end{cases}
\end{equation*}
where $ \me $ and $ \me_1 $ are two constants to be determined, and $ ( L_\ast, 1+\sigma f(L_\ast))$ is the intersecting point of the shock front $ \ms $ with the nozzle wall $ \Gamma $. Next, we need to  verify  that $ \tilde y_2(x,r) $ is well-defined and belongs to Lip($ \overline\mi $). Indeed,  using the first equation in \eqref{1-4}, one obtains
\begin{equation*}
\frac{\de \tilde y_2} {\de r}(\xi(r)+ 0,r)=\frac{\de \tilde y_2} {\de r}(\xi(r)- 0,r).
\end{equation*}
This yields $ \me_1=\me $, which can be computed as follows
\begin{equation*}
 \me^2=\int_{0}^{1}s\rho^- u_x^-(L_1,s)\de s>0.
 \end{equation*}
 \par Define the modified Lagrangian transformation   as
\begin{equation}\label{2-2}
\begin{cases}
y_1=x,\\
y_2=\tilde y_2^{\frac{1}{2}}(x,r).\\
\end{cases}
\end{equation}
If $ (\rho^\pm,u_x^\pm, u_r^\pm,u_{\theta}^\pm) $ are close to the background solution
$ (\rho_b^\pm,u_b^\pm,0,0) $, there exist two positive constants $ C_1 $ and $ C_2  $, depending on the background solution, such that
\begin{equation*}
C_1r^2\leq\tilde y_2(x,r)=\int_{0}^{r}s\rho^- u_x^-(L_1,s)\de s\leq C_2r^2.
\end{equation*}
 Then one gets $ \sqrt{C_1}r \leq y_2(x,r)\leq\sqrt{C_2}r $.   Therefore,  the Jacobian of the modified  Lagrangian transformation satisfies
 \begin{equation}\label{2-3}
\frac{\p(y_1,y_2)}{\p(x,r)}=\left|
\begin{matrix} 1& 0 \\ -\frac{r\rho u_r}{2y_2} & \frac{r\rho u_x}{2y_2} \end{matrix}\right|=\frac{r\rho u_x}{2y_2}\geq C_3>0.
\end{equation}
Hence the modified  Lagrangian transformation is invertible.
\par Under the  modified  Lagrangian transformation,   the shock front $ \ms $ and the flows before and behind $ \ms $ are denoted by $ y_1=\psi(y_2) $ and $ (u_x^\pm,u_r^\pm,u_{\theta}^\pm,P^\pm)(y_1,y_2) $ respectively.
Then the domains  $ \mi_- $ and $ \mi_+ $ are changed into
\begin{equation*}
\begin{cases}
\tilde \mi_-=\{(y_1,y_2):L_1<y_1<\psi(y_2),y_2\in[0,\me)\},\\
\tilde \mi_+=\{(y_1,y_2):\psi(y_2)<y_1<L_2,y_2\in[0,\me)\}.\\
\end{cases}
\end{equation*}
The nozzle wall $ \Gamma_{p,s} $ is straightened to be
 \begin{equation}\label{2-q}
  \Gamma_{p,y}=\{(y_1,y_2):\psi(\me) \leq y_1\leq L_2,\ y_2=\me\}.
\end{equation}
\par On the other hand, under the modified Lagrangian transformation, $ r $ as a function of $(y_1,y_2) $ becomes nonlinear and nonlocal. In fact, it follows from the inverse transformation that
\begin{equation*}
  \frac{\p r}{\p y_1}=\frac{u_r}{u_x},\quad  \frac{\p r}{\p y_2}=\frac{2y_2}{r \rho u_x}, \quad r(y_1,0)=0.
  \end{equation*}
  Then it holds that
\begin{equation}\label{2-4}
r(y_1,y_2)=2\left(\int_{0}^{y_2}\frac{s}{ \rho u_x(y_1,s)}\de s\right)^{\frac{1}{2}}.
\end{equation}
In particular, for the background solution $ (\rho_b^+,u_b^+,0,0) $, one has
  \begin{equation}\label{2-5}
r_b(y_2)=\kappa_b y_2,
\end{equation}
where $ \kappa_b=\left(\frac{2}{(\rho_b^+ u_b^+ )(y_1)}\right)^{\frac{1}{2}} $ is a positive  constant for any $ y_1\in[L_b,L_2]$.
\par For simplicity of the notations, we neglect the superscript"+"for the solution in the subsonic region. Under the transformation \eqref{2-2}, the  system \eqref{1-4} becomes
 \begin{equation}\label{2-6}
  \begin{cases}
  \begin{aligned}
  &\p_{y_1}\left(\frac{2y_2}{r\rho u_x}\right)-\p_{y_2}\left(\frac{u_r}{u_x}\right)=0,\\
  &\p_{y_1}\left(u_x+\frac{P}{\rho u_x}\right)-\frac{r}{2y_2}\p_{y_2}
  \left(\frac{Pu_r}{u_x}\right)-\frac{Pu_r}{r\rho u_x^2}=\p_{y_1}\Phi-\frac{r\rho u_r}{2y_2u_x}\p_{y_2}\Phi,\\
  &\p_{y_1}u_r+\frac{r}{2y_2}\p_{y_2}P-\frac{u_{\theta}^2}{ru_x}=\frac{r\rho }{2y_2}\p_{y_2}\Phi,\\
  &\p_{y_1}(ru_{\theta})=0.\\
  \end{aligned}
  \end{cases}
  \end{equation}
 The Rankine-Hugoniot conditions \eqref{1-15} across the shock front $ \ms $  become
\begin{equation}\label{2-7}
\begin{cases}
\begin{aligned}
&\frac{2y_2}{r}\left[\frac{1}{\rho u_x}\right]+\psi^{\prime}(y_2)
\left[\frac{ u_r}{ u_x}\right]=0,\\
&\left[u_x+\frac{ P}{ \rho u_x}\right]+\psi^{\prime}(y_2)\frac{r}{2y_2}\left[\frac{ Pu_r}{ u_x}\right]=0,\\
&[u_r]-\psi^{\prime}(y_2)\frac{r}{2y_2}[P]=0,\\
&[u_{\theta}]=0.\\
\end{aligned}
\end{cases}
\end{equation}

\subsection{The deformation-curl decomposition for the  steady   axisymmetric Euler system with an external force}\noindent
\par   Note that the steady   axisymmetric Euler system with an external force is hyperbolic-elliptic mixed in the subsonic region and the shock front is a free boundary which should be determined with the subsonic flow
simultaneously. We utilize  the deformation-curl decomposition  introduced  in \cite{WX19,WS19} to  decouple the hyperbolic and elliptic modes and find an elaborate decomposition of the Rankine-Hugoniot conditions which determines the shock front uniquely by an algebraic equation and yields suitable boundary conditions for the hyperbolic and elliptic modes, respectively.
\par Define the Bernoulli's function $ B $ by
\begin{equation}\label{2-8}
B=\frac{1}{2}{|\textbf{u}|^2} +\frac{\gamma P}{(\gamma  - 1)\rho}-\Phi.
\end{equation}
There holds
\begin{equation}\label{2-a}
\rho \mathbf{u}\cdot B=0.
\end{equation}
Using the  Bernoulli's quantity,  the density $ \rho $  can be represented as
\begin{equation}\label{2-9}
\rho=H(B,\Phi,|\textbf{u}|^2)=
\left(\frac{\gamma-1}{\gamma}(B+\Phi-\frac{1}{2}|\textbf{u}|^2)\right)
^{\frac{1}{\gamma-1}}.
\end{equation}
 Define the vorticity $ \omega=\curl \textbf{u}=\omega_x \mathbf{e}_x + \omega_r \mathbf{e}_r + \omega_{\theta} \mathbf{e}_{\theta}$, where
 \begin{equation*}
 \omega_x= \frac{1}{r}\p_r(r u_{\theta}),\ \omega_r=-\p_x u_{\theta},\ \omega_{\theta}= \p_x u_r-\p_r u_x.
 \end{equation*}
   One can follow from the third equation in \eqref{1-4} to  derive that
\begin{equation}\label{2-b}
\omega_{\theta}= \frac{u_{\theta} \omega_x}{ u_x}-\frac{\p_r B}{u_x}.
\end{equation}
\par Substituting \eqref{2-9} into the density equation and combining \eqref{2-a} and \eqref{2-b}, the system \eqref{1-4}  is  equivalent to the following  system
\begin{equation}\label{2-c}
\begin{cases}
\begin{aligned}
&(c^2(\rho)-u_x^2)\p_xu_x+(c^2(\rho)-u_r^2)\p_r u_r-u_xu_r(\p_xu_r+\p_ru_x)+u_r\frac{c^2(\rho)+u_{\theta}^2}{r}\\
&+(u_x\p_x\Phi+u_r\p_r\Phi)=0,\\
&u_x(\p_xu_r-\p_ru_x)=u_{\theta}\p_ru_{\theta}+\frac{u_{\theta}^2}{r}-\p_rB,\\
&(u_x\p_x+u_r\p_r)(ru_{\theta})=0,\\
&(u_x\p_x+u_r\p_r)B=0.\\
\end{aligned}
\end{cases}
\end{equation}
Under the modified Lagrangian transformation, the system \eqref{2-c} can be rewritten as
\begin{equation}\label{2-10}
\begin{cases}
(c^2(\rho)-u_x^2)\left(\p_{y_1}u_x-\frac{r\rho u_r}{2y_2}\p_{y_2}u_x\right)
+(c^2(\rho)-u_r^2)\left(\frac{r\rho u_x}{2y_2}\p_{y_2}u_r\right)+\frac{c^2(\rho)+u_{\theta}^2}{r}u_r\\
=-u_x\p_{y_1}\Phi
+u_xu_r\left(\p_{y_1}u_r-\frac{r\rho u_r}{2y_2}\p_{y_2}u_r+\frac{r\rho u_x}{2y_2}\p_{y_2}u_x\right),\\
u_x\left(\p_{y_1}u_r-\frac{r\rho u_r}{2y_2}\p_{y_2}u_r-\frac{r\rho u_x}{2y_2}\p_{y_2}u_x\right)
=\frac{r\rho u_x}{2y_2}u_{\theta}\p_{y_2}u_{\theta}+\frac{u_{\theta}^2}{r}-\frac{r\rho u_x}{2y_2}\p_{y_2}B,\\
\p_{y_1}(ru_{\theta})=0,\\
\p_{y_1}B=0.\\
\end{cases}
\end{equation}
\subsection{The reformulation of the Rankine-Hugoniot conditions and boundary conditions}\noindent
\par Define
\begin{equation*}
\begin{aligned}
&v_1(y_1,y_2)= u_x(y_1,y_2)- u_b^+(y_1),\quad v_2(y_1,y_2)= u_r(y_1,y_2),\quad v_3(y_1,y_2)=u_{\theta}(y_1,y_2),\\
 &v_4(y_1,y_2)= B(y_1,y_2)-B_b^+(y_1),\quad\
 {\bf v}=(v_1,v_2,v_3,v_4),\qquad\ \
v_5(y_2)=\psi(y_2)-L_b.\\
\end{aligned}
\end{equation*}
Then the density and the pressure can be expressed as
\begin{equation}\label{2-15}
\begin{aligned}
&\rho(y_1,y_2)=\rho( {\bf v})=\bigg(\frac{\gamma-1}{\gamma }\bigg)^{\frac{1}{\gamma-1}}
\bigg({B}_b^++v_4+\Phi_b+\sigma \Phi_e-\frac12(u_{b}^++v_1)^2-\sum_{j = 2}^3 |v_j|^2\bigg)^{\frac{1}{\gamma-1}},\\
&P(y_1,y_2)=P( {\bf v})=\bigg(\frac{\gamma-1}{\gamma }\bigg)^{\frac{\gamma}{\gamma-1}}
\bigg({B}_b^++v_4+\Phi_b+\sigma \Phi_e-\frac12(u_{b}^++v_1)^2-\sum_{j = 2}^3 |v_j|^2\bigg)^{\frac{\gamma}{\gamma-1}}.\\
\end{aligned}
\end{equation}
 By the third equation in \eqref{2-7}, one derives
\begin{equation}\label{2-16}
\psi^{\prime}(y_2)=\frac{2y_2[u_r]}{r[P]}=
a_1 v_2(\psi(y_2),y_2) +h_1(\bm{\Psi}^-(L_b+v_5,y_2) - {\bm{\Psi}}_b^-(L_b+v_5),{\bf v}(\psi,y_2), v_5).
\end{equation}
where $a_1= \frac{1}{\kappa_b[{P}_b(L_b)]}>0$ and
\begin{equation*}
\begin{aligned}
&h_1(\bm{\Psi}^-(L_b+h_5,y_2) - {\bm{\Psi}}_b^-(L_b+h_5),{\bf v}(\psi,y_2), h_5)\\
&=\frac{2y_2[u_r]}{r[P]}-a_1 v_2(\psi(y_2),y_2)
=v_2\left(\frac{2y_2 }{r[ P]}-a_1\right)-\frac{ 2y_2u_r^-(\psi(y_2),y_2)}{r[ P]}.
\end{aligned}
\end{equation*}
The function $h_1$ is regarded as the error term which can be bounded by
\begin{equation}\label{2-17}
|h_1|\leq C\left(|\bm{\Psi}^-(L_b+v_5,y_2) - {\bm{\Psi}}_b^-(L_b+v_5) |+|{\bf v}(\psi,y_2)|^2+|v_5|^2\right).
\end{equation}
Using the equation \eqref{2-16}, we can eliminate the quantity $ \psi^{\prime} $ in the first two equations of \eqref{2-7} to obtain
\begin{equation}\label{2-18}\begin{cases}
\begin{aligned}
&[\frac{1}{\rho u_x}] + \frac{[u_r]}{[P]} \left[\frac{u_r}{u_x}\right]=0,\\
&[u_x+\frac{P}{\rho u_x}]+ \frac{[u_r]}{[P]} \left[\frac{P u_r}{u_x}\right]=0.\\
\end{aligned}
\end{cases}\end{equation}
\par Next, a simple calculation gives
\begin{equation}\label{2-19}
\begin{cases}
\begin{aligned}
&[\rho u_x]=\rho u_x\rho^- u_x^-\frac{[u_r]}{[P]}\frac{[u_r]}{[u_x]},\\
&[\rho u_x^2+P]=-\rho^- u_x^-\frac{[u_r]}{[P]}\frac{[Pu_r]}{[u_x]}+(\rho u_x^2+P)\rho^- u_x^-\frac{[u_r]}{[P]}\frac{[u_r]}{[u_x]}.\\
\end{aligned}
\end{cases}
\end{equation}
Denote $\dot{\rho}(y_1,y_2)= \rho(y_1,y_2)-{\rho}_b^+(y_1)$. Then the first equation in \eqref{2-19} implies that
\begin{equation*}
\begin{aligned}
&{\rho}_b^+(L_b) v_1(\psi,y_2)+{ u}_{b}^+(L_b) \dot{\rho}(\psi,y_2)\\
&=-[{\rho}_b { u}_{b}](\psi)+\rho u_x\rho^- u_x^-\frac{[u_r]}{[P]}\frac{[u_r]}{[u_x]}
+(\rho^- u_x^-)(\psi,y_2)-({\rho}_b^- { u}_{b}^-)(\psi)\\
&\quad-(u_x+u_{b}^+(L_b+v_5)
-u_{b}^+(L_b))\dot{\rho}(\psi,y_2)
-({\rho}_b^+(L_b+v_5)-{\rho}_b^+(L_b))v_1(\psi,y_2)\\
&:=R_{11}(\bm{\Psi}^-(L_b+v_5,y_2) - {\bm{\Psi}}_b^-(L_b+v_5),{\bf v}(\psi,y_2), v_5).
\end{aligned}
\end{equation*}
Similarly,  one can follow from the second equation in \eqref{2-19} that at $ (\psi(y_2),y_2)$, there holds
\begin{equation}\label{2-20}
\begin{cases}
\begin{aligned}
&{\rho}_b^+(L_b) v_1(\psi,y_2)+{ u}_{b}^+(L_b) \dot{\rho}(\psi,y_2)=
R_{11}(\bm{\Psi}^-(L_b+v_5,y_2) - {\bm{\Psi}}_b^-(L_b+v_5),{\bf v}(\psi,y_2), v_5),\\
&2({\rho}_b^+ {u}_{b}^+)(L_b)v_1(\psi,y_2)+ \left(({u}_{b}^+(L_b))^2+c^2({\rho}_b^+(L_b))\right)
\dot{\rho}(\psi,y_2)\\
&=-(( \rho_{b}^+-\rho_{b}^- )g)(L_b)v_5
+R_{12}(\bm{\Psi}^-(L_b+v_5,y_2) - {\bm{\Psi}}_b^-(L_b+v_5),{\bf v}(\psi,y_2), v_5),\\
\end{aligned}
\end{cases}
\end{equation}
where
\begin{equation*}
\begin{aligned}
R_{12}&=-\left\{[\rho_b {u}_{b}^2+ {P}_b](L_b +v_5)-(( \rho_{b}^+-\rho_{b}^- )g)(L_b)v_5\right\}+ (\rho^- (u_x^-)^2 +P^-)(\psi,y_2)-({\rho}_b^-({u}_{b}^-)^2+{P}_{b}^-)(\psi)\\
&\quad-\bigg\{(\rho u_x^2 +P)(\psi,y_2)-({\rho}_b^+({u}_{b}^+)^2+{P}_{b}^+)(\psi)-2({\rho}_b^+ {u}_{b}^+)(L_b)v_1\\
&\quad- \{({u}_{b}^+(L_b))^2+c^2({\rho}_b^+(L_b))\}\dot{\rho}
\bigg\}
-\rho^- u_x^-\frac{[u_r]}{[P]}\frac{[Pu_r]}{[u_x]}+(\rho u_x^2+P)\rho^- u_x^-\frac{[u_r]}{[P]}\frac{[u_r]}{[u_x]}.\\
\end{aligned}
\end{equation*}
 Note that
\begin{equation*}
\frac{d}{d x}(  \rho_b  u_b)(x)=0, \quad \frac{d}{d x}( \rho_b  u_b^2+ P_b)(x)=( \rho_b g)(x).
\end{equation*}
Then one has
\begin{equation*}
 [ \rho_b  u_b](L_b+v_5)=O(v_5^2),\quad [ \rho_b  u_b^2+  P_b]
 (L_b+v_5)-(( \rho_b^+- \rho_b^-) g)(L_b)v_5=O(v_5^2).
 \end{equation*}
  Therefore, there exists a constant $C>0$ depending only on the background solution such that
\begin{equation}\label{2-21}
|R_{1i}|\leq C\left(|\bm{\Psi}^-(L_b+v_5,y_2) - {\bm{\Psi}}_b^-(L_b+v_5) |+|{\bf v}(\psi,y_2)|^2+|v_5|^2\right).
\end{equation}
\par By solving the algebraic equations \eqref{2-20}, one derives
\begin{equation}\label{2-22}
\begin{cases}
  \dot{\rho}(\psi, y_2)  = b_1 v_5 (y_2) + R_1 (\bm{\Psi}^-(L_b+v_5,y_2) - {\bm{\Psi}}_b^-(L_b+v_5),{\bf v}(\psi,y_2), v_5),\\
   v_1 (\psi, y_2) = b_2 v_5 (y_2) + R_2 (\bm{\Psi}^-(L_b+v_5,y_2) - {\bm{\Psi}}_b^-(L_b+v_5),{\bf v}(\psi,y_2), v_5), \\
  \end{cases}
\end{equation}
where
\begin{equation*}
b_1 =  - \frac {(\rho_b^+(L_b)-\rho_b^-(L_b))g (L_b) }
{ c^2 ({\rho}_b^+ (L_b)) -  ({u}_b^+ (L_b))^2   }  <0, \quad
b_2  = \frac {{u}_b^+ (L_b)(\rho_b^+(L_b)-\rho_b^-(L_b)){g}_b (L_b) }
{ {\rho}_b^+ (L_b)( c^2 ({\rho}_b^+ (L_b)) - ({u}_b^+ (L_b))^2)  }  >0,
\end{equation*}
and
\begin{equation*}
\begin{aligned}
R_1  = & \frac {-2{u}_b^+ (L_b) R_{11} + R_{12}}{(c^2 ({\rho}_b^+ (L_b)) - ( {u}_b^+ (L_b))^2)} , \\
R_2 = & \frac {\left(( u_b^+(L_b))^2+c^2(  \rho_b^+(L_b))\right) R_{11} - {u}_b^+ (L_b)R_{12}}
{ {\rho}_b^+(L_b) ( c^2 ({\rho}_b^+ (L_b)) - ({u}_b^+ (L_b))^2)  }.
\end{aligned}
\end{equation*}
 Furthermore, it follow from the  Bernoulli's quantity and \eqref{2-22} that
\begin{equation}\label{2-23}
v_4 (\psi, y_2)  = b_3 v_5 (y_2) + R_3 (\bm{\Psi}^-(L_b+v_5,y_2) - {\bm{\Psi}}_b^-(L_b+v_5),{\bf v}(\psi,y_2), v_5),
\end{equation}
where
\begin{equation*}
\begin{aligned}
b_3  &=  \frac{(\rho_b^-(L_b)-\rho_b^+(L_b))g (L_b) }{\rho_b^+(L_b)}  < 0,\quad R_3= \frac {- {u}_b^+ (L_b) R_{11} + R_{12}}{{\rho}_b^+ (L_b)} + R_{13},\\
R_{13} & = ( u_b^+(L_b+v_5)- u_b^+(L_b))v_1(\psi,y_2) -\frac{c^2(\rho_b^+(L_b))}{\rho_b^+(L_b)}
\dot \rho(\psi,y_2) \\
& \quad +
\frac{1}{2}\sum_{i=1}^3v_j^2(\psi,y_2) +\frac{\gamma}{\gamma-1}(\rho^+(\psi,y_2)^{\gamma - 1}
-(\rho_b^+(\psi))^{\gamma - 1}) - \sigma \Phi_e (\psi,y_2). \\
\end{aligned}
\end{equation*}
\par Next, the superscript ``+" in $\rho_b^+,{u}_{b}^+, {P}_b^+,B_b^+ $ will be ignored to simplify the notations. To derive the boundary condition at the exit, it follows from  the definition of the Bernoulli's quantity and \eqref{1-8} that
\begin{equation}\label{2-24}
\begin{aligned}
v_1(L_2,y_2)=\frac{v_4(L_2,y_2)}{{u}_{b}(L_2)}
- \frac{\sigma P_{ex}(r(L_2,y_2))}{({\rho}_b{u}_{b})(L_2)}
-\frac{1}{2u_{b}(L_2)}\sum_{j=1}^3 v_j^2(L_2,y_2)-\frac{E({\bf v}(L_2,y_2))}{u_{b}(L_2)},
\end{aligned}
\end{equation}
where
\begin{equation*}
\begin{aligned}
E({\bf v}(y_1,y_2))&=\frac{\gamma}{\gamma-1}
(P({\bf v}))^{\frac{\gamma-1}{\gamma}}-\frac{\gamma}{\gamma-1}
{P}_b^{\frac{\gamma-1}{\gamma}}
-\frac{1}{{\rho}_b(y_1)}(P({\bf v})- {P}_b)-\sigma \Phi_e.
\end{aligned}
\end{equation*}
The boundary condition on the nozzle wall is
\begin{equation}\label{2-25}
v_2(y_1,\me)=\sigma f^{\prime}(y_1)( u_b(y_1)+v_1(y_1,\me)),  \ \ {\rm{on}}\ \  \Gamma_{p,y}.
\end{equation}
\par Finally, we derive the equations for $v_j$ $(j=1,\cdots,4)$. Note that
\begin{equation*}
\begin{cases}
(c^2 (\rho_b) - {u}_b^2) {u}_b'  =- {u}_bg , \\
 c^2 (\rho ) - u^2 - c^2 (\rho_b) + {u}_b^2
  =
(\gamma - 1) ( v_4 + \sigma\Phi_e)  - \frac {\gamma + 1}{2} v_1^2 - (\gamma + 1) {u}_b v_1
 - \frac {\gamma -1 }2 \sum_{j = 2}^3 v_j^2.
 \end{cases}
 \end{equation*}
 Then it follows from  \eqref{2-10} that
 \begin{equation}\label{2-26}
\begin{cases}
\begin{aligned}
&d_1(y_1)\p_{y_1}v_1
+d_2(y_1)\bigg(\p_{y_2}v_2+\frac{v_2}{y_2}\bigg)+ d_3(y_1)v_1+d_4(y_1)v_4=F_1({\bf v}),\\
&\p_{y_1}v_2-d_2(y_1)\p_{y_2}v_1
+d_5(y_1)\p_{y_2}v_4=F_2({\bf v}),\\
&\p_{y_1}(rv_3)=0,\\
&\p_{y_1}v_4=0,\\
\end{aligned}
\end{cases}
\end{equation}
where
\begin{equation*}
\begin{aligned}
&d_1(y_1)=1-M_b^2(y_1)>0, \ \ M_b^2(y_1)=\frac{u_b^2(y_1)}{c^2(\rho_b(y_1))}, \ \
d_2(y_1)=\frac{1}{\kappa_b(y_1)}>0, \ \ d_4(y_1)=\frac{(\gamma-1) u_b^\prime(y_1)}{c^2 ({\rho}_b(y_1)) },\\
 &d_3(y_1)=\frac { g (y_1)-(\gamma+1) (u_b u_b^\prime)(y_1)}{c^2 ({\rho}_b(y_1)) }=\frac {(1 + \gamma {M}_b^2) g (y_1)}{c^2 ({\rho}_b (y_1)) - {u}_b^2 (y_1)}>0, \ \ d_5(y_1)=\frac{1}{(\kappa_bu_b)(y_1)}>0,\\
&F_1({\bf v})=\frac{1}{c^2(\rho_b)}\bigg(-(c^2(\rho)-(u_b+v_1)^2-c^2(\rho_b)+u_{b}^2)\p_{y_1}v_1+\left(\frac{\gamma+1}{2} v_1^2+\frac{\gamma-1}{2}(v_2^2+v_3^2)\right)u_{b}^\prime\\
&\quad+(c^2(\rho)-(u_b+v_1)^2)\frac{r\rho v_2}{2y_2}\p_{y_2}v_1-c^2(\rho)\frac{r\rho (u_b+v_1)}{2y_2}\p_{y_2}v_2+c^2(\rho_b)\frac{r_b\rho_b u_b}{2y_2}\p_{y_2}v_2\\
&\quad+c^2(\rho)v_2^2\frac{r\rho (u_b+v_1)}{2y_2}\p_{y_2}v_2-\frac{c^2(\rho)}{r}v_2+\frac{c^2(\rho_b)}{r_b}v_2
-\frac{c^2(\rho)v_3^2}{r}v_2-\sigma (u_b+v_1)\p_{y_1}\Phi_e\\
&\quad+(u_b+v_1)v_2\bigg(\p_{y_1}v_2-\frac{r\rho v_2}{2y_2}\p_{y_2}v_2+\frac{r\rho (u_b+v_1)}{2y_2}\p_{y_2}v_1\bigg)\bigg),\\
&F_2({\bf v})=\frac{r\rho v_2}{2y_2}\p_{y_2}v_2+\frac{r\rho (u_b+v_1)}{2y_2}\p_{y_2}v_1-\frac{r_b\rho_b u_b}{2y_2}\p_{y_2}v_1+\frac{1}{u_b+v_1}\bigg(\frac{r\rho (u_b+v_1)}{2y_2}v_3\p_{y_2}v_3+\frac{v_3^2}{r}\\
&\quad-\frac{r\rho (u_b+v_1)}{2y_2}\p_{y_2}v_4+\frac{r_b\rho_b u_b}{2y_2}\p_{y_2}v_4\bigg).\\
\end{aligned}
\end{equation*}
Here $F_1({\bf v})$ and $F_2({\bf v})$ are  quadratic and high order terms.
\par Therefore,   solve the problem \eqref{1-4} with \eqref{1-6}, \eqref{1-7},   \eqref{1-8} and \eqref{1-11}  is equivalent to find a function $v_5$ defined on $[0,\me)$ and a vector function $(v_1,\cdots, v_4)$ defined on the ${\tilde\mi_+}:=\{(y_1,y_2): L_b +v_5(y_2)<y_1<L_2,0\leq y_2< \me\}$, which solves the system \eqref{2-26} with boundary conditions \eqref{2-16}, \eqref{2-22}-\eqref{2-25}.
\subsection{The coordinates transformation}\noindent
\par In order to deal with the free boundary value problem, it is convenient
to reduce it into a fixed boundary value problem by setting
\begin{equation}\label{2-27}
z_1=\frac{y_1-\psi(y_2)}{L_2-\psi(y_2)}(L_2-L_b) + L_b=\frac{y_1-L_b-w_5}{L_2-L_b-w_5}(L_2-L_b) + L_b,\ \quad z_2=y_2,
\end{equation}
where
\begin{equation*}
 w_5 (y_2) = \psi (y_2) - L_b.
 \end{equation*}
Then
\begin{equation*}\begin{cases}
\begin{aligned}
&y_1= z_1+\frac{L_2-z_1}{L_2-L_b}w_5=: D_0^{w_5},\\
&\p_{y_1}=\frac{L_2-L_b}{L_2-L_b-w_5(z_2)} \p_{z_1}=: D_1^{w_5},\\
&\p_{y_2}=\p_{z_2}+\frac{(y_1-L_2)\p_{z_2}w_5}{L_2-L_b-w_5}\p_{z_1}=: D_2^{w_5},
\end{aligned}
\end{cases}
\end{equation*}
and the domain $\tilde\mi_+$ becomes
\begin{equation*}
\Omega=\{(z_1,z_2): L_b< z_1 < L_2, 0\leq z_2< \me\}.
\end{equation*}
\par Set
\begin{equation*}
w_j(z_1,z_2)= v_j\left(z_1+\frac{L_2-z_1}{L_2-L_b}w_5,z_2\right), \ j=1,\cdots, 4, \ {\bf w}=(w_1,\cdots, w_4).
\end{equation*}
The functions $\rho(y_1,y_2)$ and $P(y_1,y_2)$ in \eqref{2-15} can be rewritten as
\begin{equation}\label{2-28}
\begin{aligned}
&\tilde\rho(z_1,z_2)=\bigg(\frac{\gamma-1}{\gamma }\bigg)^{\frac{1}{\gamma-1}}
\bigg({B}_b(D_0^{w_5})+w_4+\Phi_b(D_0^{w_5})+\sigma \Phi_e-\frac12(u_{b}(D_0^{w_5})+w_1)^2-\sum_{j = 2}^3 |w_j|^2\bigg)^{\frac{1}{\gamma-1}},\\
&\tilde P(z_1,z_2)=\bigg(\frac{\gamma-1}{\gamma }\bigg)^{\frac{\gamma}{\gamma-1}}
\bigg({B}_b(D_0^{w_5})+w_4+\Phi_b(D_0^{w_5})+\sigma \Phi_e-\frac12(u_{b}(D_0^{w_5})+w_1)^2-\sum_{j = 2}^3 |w_j|^2\bigg)^{\frac{\gamma}{\gamma-1}}.\\
\end{aligned}
\end{equation}
 Furthermore, after the coordinate transformation, \eqref{2-16}  is changed to be
\begin{equation}\label{2-29}
w_5^{\prime}(z_2)=
a_1 w_2(L_b,z_2) +h_1(\bm{\Psi}^-(L_b+w_5,z_2) - {\bm{\Psi}}_b^-(L_b+w_5),{\bf w}(L_b,z_2), w_5).
\end{equation}
\par  In the $z$-coordinates, the transonic shock problem can be reformulated as follows. By  the second equation in \eqref{2-22}, the shock front will be determined as follows
\begin{equation}\label{2-30}
w_5 (z_2) = \frac{1}{b_2}w_1(L_b,z_2)- \frac{1}{b_2} R_1(\bm{\Psi}^-(L_b+w_5,z_2) - {\bm{\Psi}}_b^-(L_b+w_5),{\bf w}(L_b,z_2), w_5).
\end{equation}

The function $w_3$ will be determined by the following equation
\begin{equation}\label{2-31}
\begin{cases}
\p_{z_1}(\tilde rw_3)=0,\\
w_3(L_b,z_2)=u_{\theta}^-(L_b+w_5(z_2),z_2),
\end{cases}
\end{equation}
where
\begin{equation*}
\tilde r(z_1,z_2)=2\bigg(\int_{0}^{z_2}\frac{s}{u_b(D_0^{w_5}) \tilde \rho(z_1,s) +(\tilde \rho w_1)(z_1,s)}\de s\bigg)^{\frac{1}{2}}.
\end{equation*}
  The Bernoulli's quantity $w_4$ will be determined by \eqref{2-23}. That is
\begin{equation}\label{2-32}
\begin{cases}
\p_{z_1} w_4=0,\\
w_4(L_b,z_2)= b_3 w_5 (z_2) + R_3 (\bm{\Psi}^-(L_b+w_5,z_2) - {\bm{\Psi}}_b^-(L_b+w_5),{\bf w}(L_b,z_2), w_5).
\end{cases}
\end{equation}
\par Next, the first and second equations in \eqref{2-26} can be rewritten as
\begin{equation}\label{2-33}
\begin{cases}
\begin{aligned}
&d_1(z_1)\p_{z_1}w_1
+d_2(z_1)\bigg(\p_{z_2}w_2+\frac{w_2}{z_2}\bigg)+ d_3(z_1)w_1+d_4(z_1)w_4=F_3({\bf w},w_5),\\
&\p_{z_1}w_2-d_2(z_1)\p_{z_2}w_1
+d_5(z_1)\p_{z_2}w_4=F_4({\bf w},w_5),\\
\end{aligned}
\end{cases}
\end{equation}
where
\begin{equation*}
\begin{aligned}
F_3({\bf w},w_5)&=\frac{1}{c^2( \rho_b)}\bigg(-(c^2(\tilde \rho)-(u_b+w_1)^2-c^2( \rho_b)+u_{b}^2)D_1^{w_5}w_1+\left(\frac{\gamma+1}{2} w_1^2+\frac{\gamma-1}{2}(w_2^2+w_3^2)\right)u_{b}^\prime\\
&\quad+(c^2(\tilde \rho)-(u_b+w_1)^2)\frac{\tilde r\tilde \rho w_2}{2y_2}D_2^{w_5}w_1-c^2(\tilde \rho)\frac{\tilde r\tilde \rho (u_b+w_1)}{2y_2}D_2^{w_5}w_2+c^2( \rho_b)\frac{ r_b \rho_b u_b}{2y_2}D_2^{w_5}w_2\\
&\quad+c^2(\tilde \rho)w_2^2\frac{\tilde r\tilde \rho (u_b+w_1)}{2y_2}D_2^{w_5}w_2-\frac{c^2(\tilde \rho)}{\tilde r}w_2+\frac{c^2( \rho_b)}{r_b}w_2
-\frac{c^2(\tilde \rho)w_3^2}{\tilde r}w_2-\sigma (u_b+w_1)D_1^{w_5}\Phi_e\\
\end{aligned}
\end{equation*}
\begin{equation*}
\begin{aligned}
&\quad+(u_b+w_1)w_2\bigg(D_1^{w_5}w_2-\frac{\tilde r\tilde \rho w_2}{2y_2}D_2^{w_5}w_2+\frac{\tilde r\tilde \rho (u_b+w_1)}{2y_2}D_2^{w_5}w_1\bigg)\bigg)-(d_1(D_0^{w_5})D_1^{w_5}w_1-d_1(z_1)\p_{z_1}w_1)\\
&\quad
-\bigg(d_2(D_0^{w_5})\bigg(D_2^{w_5}w_2+\frac{w_2}{D_0^{w_5}}\bigg)-
d_2(z_2)\bigg(\p_{z_2}w_2+\frac{w_2}{z_2}\bigg)\bigg)-( d_3(D_0^{w_5})- d_3(z_1))w_1\\
&\quad-( d_4(D_0^{w_5})- d_4(z_1))w_1,\\
F_4({\bf w},w_5)&=\frac{\tilde r\tilde \rho w_2}{2y_2}D_2^{w_5}w_2+\frac{\tilde r\tilde \rho (u_b+w_1)}{2y_2}D_2^{w_5}w_1-\frac{r_b \rho_b u_b}{2y_2}D_2^{w_5}w_1+\frac{1}{u_b+w_1}\bigg(\frac{\tilde r\tilde \rho (u_b+w_1)}{2y_2}w_3D_2^{w_5}w_3+\frac{w_3^2}{\tilde r}\\
&\quad-\frac{\tilde r\tilde \rho (u_b+w_1)}{2y_2}D_2^{w_5}w_4+\frac{ r_b \rho_b u_b}{2y_2}D_2^{w_5}w_4\bigg)-(D_1^{w_5}w_2-\p_{z_1}w_2)+(d_2(D_0^{w_5})D_2^{w_5}w_2-
d_2(z_1)\p_{z_2}w_1)\\
&\quad-(d_5(D_0^{w_5})D_2^{w_5}w_4-d_5(z_1)\p_{z_2}w_4).\\
\end{aligned}
\end{equation*}
\par Finally,
the boundary condition \eqref{2-24} at the exit becomes
\begin{equation}\label{2-34}
\begin{aligned}
w_1(L_2,z_2)=\frac{w_4(L_2,z_2)}{{u}_{b}(L_2)}
- \frac{\sigma P_{ex}(\tilde r(L_2,z_2))}{({\rho}_b{u}_{b})(L_2)}
-\frac{1}{2u_{b}(L_2)}\sum_{j=1}^3 w_j^2(L_2,z_2)-\frac{E({\bf w}(L_2,z_2))}{u_{b}(L_2)},
\end{aligned}
\end{equation}
where
\begin{equation*}
\begin{aligned}
E({\bf w}(z_1,z_2))&=\frac{\gamma}{\gamma-1}
(\tilde P({\bf w}))^{\frac{\gamma-1}{\gamma}}-\frac{\gamma}{\gamma-1}
{P}_b^{\frac{\gamma-1}{\gamma}}
-\frac{1}{{\rho}_b(z_1)}(\tilde P({\bf w})- {P}_b)-\sigma \Phi_e.
\end{aligned}
\end{equation*}
The boundary condition on the nozzle wall is
\begin{equation}\label{2-35}
w_2(z_1,\me)=\sigma f^{\prime}(D_0^{w_5})( u_b(D_0^{w_5})+w_1(z_1,\me)),  \ \ {\rm{on}}\ \  \Gamma_{p,z}=\{(z_1,z_2):L_b \leq z_1\leq L_2,\ z_2=\me\}.
\end{equation}

 \par Therefore, after the coordinates transformation \eqref{2-27}, the problem  \eqref{2-26} with boundary conditions \eqref{2-16}, \eqref{2-22}-\eqref{2-25}   is equivalent to solve the following problem.
\begin{problem}
Find a function $w_5$ defined on $[0,\me)$ and  a vector function $(w_1,\cdots, w_4)$ defined on the $\Omega$, which solves \eqref{2-31}-\eqref{2-32} and \eqref{2-33} with boundary conditions \eqref{2-29}, \eqref{2-30}, \eqref{2-34} and \eqref{2-35}.
\end{problem}
Theorem 1.3 then follows directly from the following result.
\begin{theorem}
Assume that the compatibility conditions \eqref{1-5} and \eqref{1-13} hold. There exist suitable positive constants $ \sigma_0 $ and $ C_\ast$ depending only on the background solution $ \mathbf{\Psi}_b $ defined in \eqref{1-3} and the boundary data $ \mathbf{\Psi}^-(L_1,\cdot) $, $ f $, $ P_{ex} $ ,$ \Phi_e $ such that if $ 0< \sigma \leq \sigma_0 $, the problem \eqref{2-31}-\eqref{2-33}  with boundary conditions \eqref{2-29}, \eqref{2-30}, \eqref{2-34} and \eqref{2-35}  has a unique  solution $  (w_1,w_2,w_3,w_4)(z_1,z_2) $ with the shock front $ \ms: z_1=w_5(z_2) $ satisfying the following properties.
\begin{item}[\rm{(1)}]
The function $ w_5(z_2) \in C_{3,\alpha}^{(-1-\alpha;\{\me\})}([0,\me)) $ satisfies
\begin{equation}\label{2-36}
\|w_5\|_{3,\alpha;[0,\me)}^{(-1-\alpha;\{\me\})} \leq C_\ast\sigma,
\end{equation}
and
\begin{equation}\label{2-37}
w_5^\prime(0)=w_5^{(3)}(0)=0.
\end{equation}
 \end{item}
\begin{item}[\rm{(2)}]
The solution $ (w_1,w_2,w_3,w_4)(z_1,z_2)\in \left(C_{2,\alpha}^{(-\alpha;\Gamma_{p,z})}(\Omega)\right)^4 $ satisfies    the estimate
\begin{equation}\label{2-38}
\sum_{i=1}^4\|w_i\|_{2,\alpha;\Omega}^{(-\alpha;\Gamma_{p,z})}\leq C_\ast\sigma
\end{equation}
 with the compatibility conditions
 \begin{equation}\label{2-39}
(w_2,{\p_{z_2}^2} w_2)(z_1,0)=(w_3, \p_{z_2}w_3)(z_1,0)= \p_{z_2}(w_1,w_4)(z_1,0)
 =0, \quad \forall z_1\in[L_b,L_2].
  \end{equation}
\end{item}
\end{theorem}
\section{The stability analysis}\noindent
\par In  this section, we first   construct a suitable iteration scheme to linearize the problem  \eqref{2-31}-\eqref{2-33} with boundary conditions \eqref{2-29}, \eqref{2-30}, \eqref{2-34} and \eqref{2-35} .  Especially,  a linear first order elliptic system with a nonlocal term  can be derived. Then one can introduce a potential function to reduce the first order elliptic system into a second order elliptic equation with a nonlocal term involving only the trace of the potential function on the shock front and   a free parameter. We  solve this  second order nonlocal elliptic  equation with  a free parameter and establish some  prior estimates and then complete the proof of Theorem 2.2.
\subsection{An iteration scheme}\noindent
\par  In this subsection, we develop an iteration scheme to linearize the problem \eqref{2-31}-\eqref{2-33} with boundary conditions \eqref{2-29}, \eqref{2-30}, \eqref{2-34} and \eqref{2-35}.
 The solution class $\mj$ consists of the vector functions $(w_1,\cdots, w_4,w_5)\in \bigg(\left(C_{2,\alpha}^{(-\alpha;\Gamma_{p,z})}(\Omega)\right)^4\times C_{3,\alpha}^{(-1-\alpha;\{\me\})}([0,\me)) \bigg)$ satisfying the estimate
\begin{equation}\label{3-1}
 \|({\bf w}, w_5)\|_{\mj}= \sum_{j=1}^4\|w_j\|_{2,\alpha;\Omega}^{(-\alpha;\Gamma_{p,z})}
+\|w_5\|_{3,\alpha;[0,\me)}^{(-1-\alpha;\{\me\})} \leq \delta,
\end{equation}
and the  compatibility conditions
\begin{equation}\label{3-2}
w_5^\prime(0)=w_5^{(3)}(0)=0, \ (w_2,{\p_{z_2}^2} w_2)(z_1,0)=(w_3, \p_{z_2}w_3)(z_1,0)= \p_{z_2}(w_1,w_4)(z_1,0)=0, \  \forall z_1\in[L_b,L_2],
 \end{equation}
    where $\delta>0$  to be determined later. Given $(\hat{\bf w}, \hat w_5)\in \mj$, we will construct an iterative procedure that generates a new $({\bf w},  w_5) \in \mj$, and thus one can define a mapping  from $\mj$ to itself by choosing a suitably small positive constant $\delta$.
   \par {\bf Step 1. The iteration scheme for $w_5$.}
\par The shock front $w_5$ is uniquely determined by
\begin{equation}\label{3-3}
w_5 (z_2) = \frac{1}{b_2}w_1(L_b,z_2)- \frac{1}{b_2} R_1(\bm{\Psi}^-(L_b+\h w_5,z_2) - {\bm{\Psi}}_b^-(L_b+\h w_5),\h{\bf  w}(L_b,z_2), \h w_5),
\end{equation}
provided that $w_1(L_b,z_2)$ is obtained.
\par {\bf Step 2. The iteration scheme for  $w_3$ and $w_4$.}
\par We solve the transport equations for the swirl velocity and the Bernoulli's quantity . The swirl velocity $w_3$ will be determined by
\begin{equation}\label{3-4}
\begin{cases}
\p_{z_1}(\hat rw_3)=0,\\
w_3(L_b,z_2)=u_{\theta}^-(L_b+\h w_5(z_2),z_2),
\end{cases}
\end{equation}
where
\begin{equation*}
\hat r(z_1,z_2)=2\bigg(\int_{0}^{z_2}\frac{s}{ u_b(D_0^{\h w_5}) \hat\rho( z_1,s)+(\hat\rho\hat w_1)(z_1,s)}\de s\bigg)^{\frac{1}{2}}.
\end{equation*}
Then $w_3$ can be solved as follows
\begin{equation}\label{3-5}
w_3(z_1,z_2)=\frac{\hat r(L_b,z_2) u_{\theta}^-(L_b+\h w_5(z_2),z_2)}
{\hat r(z_1,z_2)}.
\end{equation}
\par The Bernoulli's quantity $w_4$ satisfies
\begin{equation}\label{3-6}
\begin{cases}
\p_{z_1} w_4=0,\\
w_4(L_b,z_2)= b_3 w_5 (z_2) + R_3 (\bm{\Psi}^-(L_b+\h w_5,z_2) - {\bm{\Psi}}_b^-(L_b+\h w_5),\h{\bf w}(L_b,z_2), \h w_5).
\end{cases}
\end{equation}
This, together with \eqref{3-3}, yields that
\begin{equation}\label{3-7}
\begin{aligned}
w_4(z_1,z_2)&=b_3 w_5 (z_2) + R_3 (\bm{\Psi}^-(L_b+\h w_5,z_2) - {\bm{\Psi}}_b^-(L_b+\h w_5),\h{\bf w}(L_b,z_2), \h w_5)\\
&=\frac{b_3}{b_2}w_1(L_b,y_2)+R_4(\bm{\Psi}^-(L_b+\h w_5,z_2) - {\bm{\Psi}}_b^-(L_b+\h w_5),\h{\bf w}(L_b,z_2), \h w_5),
\end{aligned}
\end{equation}
where
\begin{equation*}
\begin{aligned}
&R_4(\bm{\Psi}^-(L_b+\h w_5,z_2) - {\bm{\Psi}}_b^-(L_b+\h w_5),\h{\bf w}(L_b,z_2), \h w_5)\\
&= R_3 (\bm{\Psi}^-(L_b+\h w_5,z_2) - {\bm{\Psi}}_b^-(L_b+\h w_5),\h{\bf w}(L_b,z_2), \h w_5)\\
&\quad- \frac{b_3}{b_2} R_1(\bm{\Psi}^-(L_b+\h w_5,z_2) - {\bm{\Psi}}_b^-(L_b+\h w_5),\h{\bf  w}(L_b,z_2), \h w_5).
\end{aligned}
\end{equation*}
\par {\bf Step 3. The iteration scheme for $w_1$ and $w_2$.}
\par We derive the equations for $w_1$ and $w_2$.  Firstly, it follows from \eqref{2-29} that
 \begin{equation}\label{3-a}
w_5^{\prime}(z_2)=
a_1 w_2(L_b,z_2) +h_1(\bm{\Psi}^-(L_b+\h w_5,z_2) - {\bm{\Psi}}_b^-(L_b+\h w_5),\h {\bf w}(L_b,z_2), \h w_5).
\end{equation}
Substituting \eqref{3-7} into \eqref{2-33} and combining \eqref{3-3}, \eqref{3-a}, \eqref{2-34} and \eqref{2-35}, one gets
\begin{equation}\label{3-8}
\begin{cases}
\begin{aligned}
&d_1(z_1)\p_{z_1}w_1
+d_2(z_1)\bigg(\p_{z_2}w_2+\frac{w_2}{z_2}\bigg)+ d_3(z_1)w_1+d_4(z_1)\frac{b_3}{b_2}w_1(L_b,z_2)=G_1(\h{\bf w},\h w_5),\\
&\p_{z_1}w_2-d_2(z_1)\p_{z_2}w_1
+d_5(z_1)\frac{b_3}{b_2}\p_{z_2}w_1(L_b,z_2)=G_2(\h{\bf w},\h w_5),\\
&\p_{z_2}w_1(L_b,z_2)=a_1b_2 w_2(L_b,z_2)+h_2(z_2),\\
&w_1(L_2,z_2)=
\frac{b_3}{b_2{u}_{b}(L_2)}w_1(L_b,z_2)+h_3(z_2),\\
&w_2(z_1,0)=0,\\
&w_2(z_1,\me )=h_4(z_1),\\
\end{aligned}
\end{cases}
\end{equation}
where
\begin{equation*}
\begin{aligned}
G_1(\h{\bf w},\h w_5)&=F_3(\h{\bf w},\h w_5)-d_4(z_1)R_4(\bm{\Psi}^-(L_b+\h w_5,z_2) - {\bm{\Psi}}_b^-(L_b+\h w_5),\h{\bf w}(L_b,z_2), \h w_5), \\
G_2(\h{\bf w},\h w_5)&=F_4(\h{\bf w},\h w_5)-d_5(z_1)\p_{z_2}R_4(\bm{\Psi}^-(L_b+\h w_5,z_2) - {\bm{\Psi}}_b^-(L_b+\h w_5),\h{\bf w}(L_b,z_2), \h w_5),\\
h_2(z_2)&=h_1 (\bm{\Psi}^-(L_b+\h w_5,z_2) - {\bm{\Psi}}_b^-(L_b+\h w_5),\h{\bf w}(L_b,z_2), \h w_5)\\
& \quad +\p_{z_2}R_1(\bm{\Psi}^-(L_b+\h w_5,z_2) - {\bm{\Psi}}_b^-(L_b+\h w_5),\h{\bf w}(L_b,z_2), \h w_5),\\
h_3(z_2)&=\frac{R_4(\bm{\Psi}^-(L_b+\h w_5,z_2) - {\bm{\Psi}}_b^-(L_b+\h w_5),\h{\bf w}(L_b,z_2), \h w_5)}{{u}_{b}(L_2)}
- \frac{\sigma P_{ex}(\hat r(L_2,z_2))}{({\rho}_b{u}_{b})(L_2)}\\
& \quad-\frac{1}{2u_{b}(L_2)}\sum_{j=1}^3 \h w_j^2(L_2,z_2)-\frac{E(\h {\bf w}(L_2,z_2))}{u_{b}(L_2)},\\
h_4(z_1)&=\sigma f^{\prime}(D_0^{\h w_5})( u_b(D_0^{\h w_5})+\h w_1(z_1,\me)).
\end{aligned}
\end{equation*}
Then it follows from the expressions of $ F_i $ and $R_i $ together with  direct computations that
\begin{equation}\label{3-9}
 \begin{aligned}
 \sum_{j=1}^2\|G_i(\h{\bf w},\h w_5)\|_{1,\alpha;\Omega}^{(1-\alpha;\Gamma_{p,z})}+ \|h_2\|_{1,\alpha;[0,\me)}^{(1-\alpha;\{\me\})}
 +\|h_3\|_{1,\alpha;[0,\me)}^{(-\alpha;\{\me\})}
 +\|h_4\|_{0,\alpha;[L_b,L_2]}
 \leq C\left(\sigma+
\delta^2\right).
 \end{aligned}
 \end{equation}

 \par Next, 
the second eqaution in \eqref{3-8} can be rewritten as
\begin{equation*}
\p_{z_1}w_2-\p_{z_2}\bigg(d_2(z_1)w_1
-d_5(z_1)\frac{b_3}{b_2}w_1(L_b,z_2)-\int_{z_2}^{\me}G_2(\h{\bf w},\h w_5)(z_1,s)\de s\bigg)=0,
\end{equation*}
which implies that there exists a potential function $ \phi $ such that
\begin{equation}\label{3-10}
\p_{z_2}\phi=w_2, \ \ \p_{z_1}\phi=d_2(z_1)w_1
-d_5(z_1)\frac{b_3}{b_2}w_1(L_b,z_2)-\int_{z_2}^{\me}G_2(\h{\bf w},\h w_5)(z_1,s)\de s, \ \ \phi(L_b,\me)=0.
\end{equation}
Therefore, one obtians
\begin{equation*}
\begin{aligned}
&w_1(L_b,z_2)=b_4\bigg(\p_{z_1}\phi(L_b,z_2)+\int_{z_2}^{\me}G_2(\h{\bf w},\h w_5)(L_b,s)\de s\bigg), \ \ b_4=d_2(L_b)
-d_5(L_b)\frac{b_3}{b_2}=\frac{c^2(\rho_b(L_b))}{\kappa_b(L_b)}>0,\\
&w_1(z_1,z_2)=\frac{1}{d_2(z_1)}\bigg(\p_{z_1}\phi
+d_5(z_1)\frac{b_3b_4}{b_2}\p_{z_1}\phi(L_b,z_2)+\int_{z_2}^{\me}(G_2(\h{\bf w},\h w_5)(z_1,s)+b_4d_5(z_1)G_2(\h{\bf w},\h w_5)(L_b,s))\de s\bigg).
\end{aligned}
\end{equation*}
Then the problem  \eqref{3-8}  is reduced to
\begin{equation}\label{3-11}
\begin{cases}
\begin{aligned}
&\p_{z_1}(\lambda_1(z_1)\p_{z_1}\phi)
+\lambda_2(z_1)\bigg(\p_{z_2}^2\phi+\frac{\p_{z_2}\phi}{z_2}\bigg)-\lambda_3(z_1)\p_{z_1}\phi(L_b,z_2)
=\p_{z_1}\mg_1+\p_{z_2}\mg_2,\\
&\p_{z_2}(\p_{z_1}\phi(L_b,z_2)-b_5\phi(L_b,z_2))=q_2(z_2),\\
&\p_{z_1}\phi(L_2,z_2)= q_3(z_2),\\
&\p_{z_2}\phi(z_1,0)=0,\\
&\p_{z_2}\phi(z_1,\me)=h_4(z_1),\\
\end{aligned}
\end{cases}
\end{equation}
where
\begin{equation*}
\begin{aligned}
&\lambda_1(z_1)=\frac{\lambda_0(z_1)}{d_2(z_1)}>0, \quad \lambda_2(z_1)=\lambda_0(z_1) \frac{d_2}{d_1}(z_1)>0, \quad \lambda_0(z_1)=\text{exp}\bigg(\int_{L_b}^{z_1} \frac{d_3}{d_1}(s) \de s\bigg),\\
&\lambda_3(z_1)=-\frac{b_3b_4}{b_2}\lambda_0(z_1)\bigg( \frac{d_2d_3}{d_1d_5}(z_1)+\left(\frac{d_5}{d_2}\right)^\prime(z_1)+
\frac{d_4}{d_1}(z_1)\bigg)\\
&\qquad\ \ =-\frac{b_3b_4}{b_2}\lambda_0(z_1)\frac{(u_bg)(z_1)}{c^2(\rho_b(z_1))
-u_b^2(z_1)}
\frac{2+
(\gamma-1)M_b^4(z_1)}{1-M_b^2(z_1)}>0,\\
&\mg_1(z_1,z_2)=-\lambda_1(z_1)\int_{z_2}^{\me}(G_2(\h{\bf w},\h w_5)(z_1,s)+b_4d_5(z_1)G_2(\h{\bf w},\h w_5)(L_b,s))\de s,\\
&\mg_2(z_1,z_2)=\frac{\lambda_0(z_1)}{d_1(z_1)}\int_{0}^{z_2}G_1(\h{\bf w},\h w_5)(z_1,s)\de s,\quad
b_5=\frac{a_1b_2}{b_4}>0, \\
&q_2(z_2)=G_2(\h{\bf w},\h w_5)(L_b,z_2)+\frac{h_2(z_2)}{b_4},\
q_3(z_2)=-\int_{z_2}^{\me}G_2(\h{\bf w},\h w_5)(L_2,s)\de s+h_3(z_2).\\
\end{aligned}
\end{equation*}
\par The second equation in \eqref{3-11}  implies that
\begin{equation}\label{3-12}
\p_{z_1}\phi(L_b,z_2)-b_5(\phi(L_b,z_2)+\Lambda)=\tilde q_2(z_2),
\end{equation}
where
\begin{equation*}
\begin{aligned}
&\Lambda=-\frac{w_5(\me)}{a_1},\\
&\tilde q_2(z_2)=-\int_{z_2}^{\me}q_2(s)\de s+\frac{1}{b_4}R_1(\bm{\Psi}^-(L_b+\h w_5(\me),\me) - {\bm{\Psi}}_b^-(L_b+\h w_5(\me)),\h{\bf  w}(L_b,\me), \h w_5(\me)).
\end{aligned}
\end{equation*}
Substituting \eqref{3-12} into the first equation in \eqref{3-11} yields
 \begin{equation}\label{3-13}
\begin{cases}
\begin{aligned}
&\p_{z_1}(\lambda_1(z_1)\p_{z_1}\phi)
+\lambda_2(z_1)\bigg(\p_{z_2}^2\phi+\frac{\p_{z_2}\phi}{z_2}\bigg)-
\lambda_3(z_1)b_5(\phi(L_b,z_2)+\Lambda)
=\p_{z_1}\mg_1+\p_{z_2}\mg_2+\mg_3,\\
&\p_{z_1}\phi(L_b,z_2)-b_5(\phi(L_b,z_2)+\Lambda)=\tilde q_2(z_2),\\
&\p_{z_1}\phi(L_2,z_2)= \tilde q_3(z_2),\\
&\p_{z_2}\phi(z_1,0)=0,\\
&\p_{z_2}\phi(z_1,\me)=h_4(z_1),\\
\end{aligned}
\end{cases}
\end{equation}
where
\begin{equation*}
\mg_3(z_1,z_2)=\lambda_3(z_1)\tilde q_2(z_2), \quad \tilde q_3(z_2)=q_3(z_2).
\end{equation*}
Furthermore, it follows from \eqref{3-9} that the following estimate holds:
 \begin{equation}\label{3-14}
 \begin{aligned}
\sum_{j=1}^3\|\mg_j\|_{1,\alpha;\Omega}^{(-\alpha;\Gamma_{p,z})}+ \sum_{j=2}^3\|\tilde q_j\|_{1,\alpha;[0,\me)}^{(-\alpha;\{\me\})}
 +\|h_4\|_{0,\alpha;[L_b,L_2]}
 \leq C\left(\sigma+
 \|(\h{\bf w}, \h w_5)\|_{\mj}^2\right).
 \end{aligned}
 \end{equation}
  \subsection{Solving a second order elliptic equation with a nonlocal term }\noindent
  \par  Let
 $\phi_\ast(z_1,z_2)=\phi(z_1,z_2)+\Lambda $. Then \eqref{3-13} is equivalent to the following problem
  \begin{equation}\label{3-15}
\begin{cases}
\begin{aligned}
&\p_{z_1}(\lambda_1(z_1)\p_{z_1}\phi_\ast)
+\lambda_2(z_1)\bigg(\p_{z_2}^2\phi_\ast+\frac{\p_{z_2}\phi_\ast}{z_2}\bigg)-
\lambda_3(z_1)b_5\phi_\ast(L_b,z_2)
=\p_{z_1}\mg_1+\p_{z_2}\mg_2+\mg_3,\\
&\p_{z_1}\phi_\ast(L_b,z_2)-b_5\phi_\ast(L_b,z_2)=\tilde q_2(z_2),\\
&\p_{z_1}\phi_\ast(L_2,z_2)=\tilde q_3(z_2),\\
&\p_{z_2}\phi_\ast(z_1,0)=0,\\
&\p_{z_2}\phi_\ast(z_1,\me)=h_4(z_1).\\
\end{aligned}
\end{cases}
\end{equation}
In order to deal with the singularity near $ z_2=0 $, we rewrite the problem \eqref{3-15} by using the cylindrical coordinates transformation again. Define
\begin{equation*}
\eta_1=z_1,\ \eta_2=z_2\cos\tau,\ \eta_3=z_2\sin\tau,\ \tau\in[0,2\pi],
\end{equation*}
and
\begin{equation*}
\begin{aligned}
&\Omega_1=\{(\eta_1,\eta_2,\eta_3):L_b<\eta_1<L_2, \eta_2^2+\eta_3^2\leq \me^2\},\quad
\Omega_2=\{(\eta_2,\eta_3):\eta_2^2+\eta_3^2\leq \me^2\},\\
&\Gamma_{\eta}^{\prime}=\{\eta^{\prime}=(\eta_2,\eta_3):\eta_2^2+\eta_3^2= \me^2\},\quad
\Gamma_{p,\eta}=[L_b,L_2]\times \Gamma_{\eta}^{\prime},\\
&\Psi(\bm \eta)=\phi_\ast(\eta_1,\sqrt{\eta_2^2+\eta_3^2})
=\phi_\ast(\eta_1,|\eta^{\prime}|).\\
\end{aligned}
\end{equation*}
Then   $ \Psi(\bm \eta) $ satisfies
\begin{equation}\label{3-16}
\begin{cases}
\begin{aligned}
&\p_{\eta_1}(\lambda_1(\eta_1)\p_{\eta_1}\Psi)+\lambda_2(\eta_1)
\sum_{j=2}^{3}\p_{\eta_j}^2\Psi
-\lambda_3(\eta_1)b_5\Psi(L_b,\eta^{\prime})\\
&=\p_{\eta_1}\mg_1(\eta_1,|\eta^{\prime}|)+
\sum_{j=2}^{3}\p_{\eta_j}\left(\frac{\eta_j\mg_2(\eta_1,|\eta^{\prime}|)}
{|\eta^{\prime}|}\right)-\frac{\mg_2(\eta_1,|\eta^{\prime}|)}
{|\eta^{\prime}|}+\mg_3(\eta_1,|\eta^{\prime}|),\\
&\p_{\eta_1}\Psi(L_b,\eta^{\prime})-b_5\Psi(L_b,\eta^{\prime})
=\tilde q_2(|\eta^{\prime}|),\\
&\p_{\eta_1}\Psi(L_2,\eta^{\prime})=\tilde q_3(|\eta^{\prime}|),\\
&(\eta_2\p_{\eta_2}+\eta_3\p_{\eta_3})\Psi(\eta_1,\eta^{\prime})
=\me h_4(\eta_1).\\
\end{aligned}
\end{cases}
\end{equation}
\par Firstly, the weak solution to \eqref{3-16} can be obtained as follows. $\Psi\in H^1(\Omega_1)$ is said to be a weak solution to \eqref{3-16}, if for any $\varphi\in H^1(\Omega_1)$, the following equality holds:
\begin{equation}\label{3-17}
\ma(\Psi,\varphi)= \mf(\varphi),\ \ \forall \varphi\in H^1(\Omega_1),
\end{equation}
where
\begin{equation*}
\begin{aligned}
\ma(\Psi,\varphi)&=\iiint_{\Omega_1}\lambda_1(\eta_1)\p_{\eta}\Psi
\p_{\eta_1}\varphi
+\lambda_2(\eta_1)(\p_{\eta_2}\Psi \p_{\eta_2}\varphi+\p_{\eta_3}\Psi \p_{\eta_3}\varphi)+\lambda_3(\eta_1)b_5\Psi(L_b,\eta^{\prime})
\varphi(\eta_1,\eta^{\prime})
\de \eta_1\de \eta_2 \de \eta_3\\
&\quad +\iint_{\Omega_2}\lambda_1(L_b)b_5\Psi(L_b,\eta^{\prime})
\varphi(L_b,\eta^{\prime})\de \eta_2 \de \eta_3,\\
\mf(\varphi)&=\iiint_{\Omega}\mg_1\p_{\eta_1}\varphi+\sum_{j=2}^{3}
\frac{\eta_j\mg_2}
{|\eta^{\prime}|}\p_{\eta_i}\varphi-\bigg(\mg_3-\frac{\mg_2}
{|\eta^{\prime}|}\bigg)\varphi\de \eta_1\de \eta_2\de \eta_3
-\int_{L_b}^{L_2} \mg_2(\eta_1,\me)\varphi(\eta_1,\me)\de \eta_1\\
&\quad-\iint_{\Omega_2}\mg_1(L_2,|\eta^{\prime}|))\varphi(L_2,\eta^{\prime})
-\mg_1(L_b,|\eta^{\prime}|))\varphi(L_b,\eta^{\prime})\de \eta_2 \de \eta_3 +\int_{L_b}^{L_2}\me h_4(\eta_1)\varphi(\eta_1,\me)\de \eta_1\\
&\quad+\iint_{\Omega_2}\lambda_1(L_2)\tilde q_3(|\eta^{\prime}|))
   \varphi(L_2,\eta^{\prime})
-\lambda_1(L_b)\tilde q_2(|\eta^{\prime}|))\varphi(L_b,\eta^{\prime})\de \eta_2 \de \eta_3
.
\end{aligned}
\end{equation*}
\begin{lemma}
 There exists a positive constant $K$ depending only on the background solution such that the following problem
 \begin{equation}\label{3-18}
\begin{cases}
\begin{aligned}
&\p_{\eta_1}(\lambda_1(\eta_1)\p_{\eta_1}\Psi)+\lambda_2(\eta_1)
\sum_{j=2}^{3}\p_{\eta_j}^2\Psi
-\lambda_3(\eta_1)b_5\Psi(L_b,\eta^{\prime})+K\Psi\\
&=\p_{\eta_1}\mg_1(\eta_1,|\eta^{\prime}|)+
\sum_{j=2}^{3}\p_{\eta_j}\left(\frac{\eta_j\mg_2(\eta_1,|\eta^{\prime}|)}
{|\eta^{\prime}|}\right)-\frac{\mg_2(\eta_1,|\eta^{\prime}|)}
{|\eta^{\prime}|}+\mg_3(\eta_1,|\eta^{\prime}|),\\
&\p_{\eta_1}\Psi(L_b,\eta^{\prime})-b_5\Psi(L_b,\eta^{\prime})
=\tilde q_2(|\eta^{\prime}|),\\
&\p_{\eta_1}\Psi(L_2,\eta^{\prime})=\tilde q_3(|\eta^{\prime}|),\\
&(\eta_2\p_{\eta_2}+\eta_3\p_{\eta_3})\Psi(\eta_1,|\eta^{\prime}|)
=\me h_4(\eta_1),\\
\end{aligned}
\end{cases}
\end{equation}
has a unique weak solution in $ H^1(\Omega_1)$.
\end{lemma}
\begin{proof}
The system \eqref{3-18} has the following bilinear form on $ H^1(\Omega_1)\times H^1(\Omega_1)$:
\begin{equation}\label{3-20}
\ma_K(\Psi,\varphi)=\ma(\Psi,\varphi)+\iiint_{\Omega_1}\Psi\varphi\de \eta_1\de \eta_2\de \eta_3= \mf(\varphi),\ \ \forall \varphi\in H^1(\Omega_1).
\end{equation}
For any $\epsilon>0 $, one can  use Cauchy's inequality to  get
\begin{equation*}
\begin{aligned}
&\iiint_{\Omega_1}\Psi(L_b,\eta^{\prime})\Psi(\eta_1,\eta^{\prime})\de \eta_1\de \eta_2\de \eta_3\\
&\leq \frac{C_1}{\epsilon}\iiint_{\Omega_1}\Psi^2(\eta_1,\eta^{\prime}))\de \eta_1\de \eta_2\de \eta_3+\epsilon\iiint_{\Omega_1}(\p_{\eta_1}\Psi)^2(\eta_1,\eta^{\prime}))
\de \eta_1\de \eta_2\de \eta_3 .
\end{aligned}
\end{equation*}
 Note that $ \mg_2(\eta_1,0)=0 $. Thus the boundedness and coercivity of $\ma_K$ can be verified as follows
\begin{equation*}
\begin{aligned}
&|\ma_K(\Psi,\varphi)|\leq C\|\Psi\|_{H^1(\Omega_1)}\|\varphi\|_{H^1(\Omega)},\\
 &|\mf(\varphi)|\leq C\bigg(\sum_{j=1}^3\|\mg_j\|_{C^{0,\alpha}(\overline{\Omega_1})}+ \sum_{j=2}^3\|\tilde q_j\|_{C^{0,\alpha}(\overline{\Omega_2})}
 +\|h_4\|_{C^{0,\alpha}[L_b,L_2]}\bigg)\|\varphi\|_{H^1(\Omega_1)},
\end{aligned}
\end{equation*}
and
\begin{equation*}
\begin{aligned}
\ma_K(\Psi,\Psi)&=\iiint_{\Omega_1}\bigg(\lambda_1(\eta_1)
(\p_{\eta_1}\Psi)^2
+\lambda_2(\eta_1)\sum_{i=2}^3(\p_{\eta_i}\Psi)^2 +\lambda_3(\eta_1)b_5\Psi(L_b,\eta^{\prime})
\Psi(\eta,\eta^{\prime})+K\Psi^2\bigg)
\de \eta_1\de \eta_2 \de \eta_3\\
& +\iint_{\Omega_2}\lambda_1(L_b)b_5(\Psi(L_b,\eta^{\prime}))^2
\de \eta_2 \de \eta_3,\\
&\geq C_*\bigg(\|\n\Psi\|_{L^2(\Omega_1)}^2
+\|\Psi(L_b,\cdot)\|_{L^2(\Omega_2)}^2\bigg)+K\|\Psi\|_{L^2(\Omega_1)}^2
-\frac{C_*}{4}\|\p_{\eta_1}\Psi\|_{L^2(\Omega_1)}^2\\
&\quad-\frac{C_*}{4}\|\Psi(L_b,\cdot)\|_{L^2(\Omega_2)}^2
-\tilde{C}_*\|\Psi\|_{L^2(\Omega_1)}^2\\
&\geq \frac{C_*}{2}\bigg(\|\n \Psi\|_{L^2(\Omega_1)}^2+\|\Psi(L_b,\cdot)\|_{L^2(\Omega_2)}^2\bigg)
+\frac{K}{2}\|\Psi\|_{L^2(\Omega_1)}^2
\end{aligned}
\end{equation*}
provided that $K$ is sufficiently large. Then the existence and uniqueness of $H^1 (\Omega_1)$ solution $\Psi$ to \eqref{3-16} can be obtained by using the Lax-Milgram theorem, which completes the proof of Lemma 3.1.
\end{proof}
\par Now we are going to solve the problem \eqref{3-16}.
\begin{proposition}
Suppose that $(\mg_1,\mg_2,,\mg_3)\in \left(C_{1,\alpha}^{(-\alpha;\Gamma_{p,\eta})}(\Omega_1)\right)^3 $,$ \mg_2(\eta_1,0)=0 $, $(\tilde q_2,\tilde q_3)\in \left(C_{1,\alpha}^{(-\alpha;\Gamma_{\eta}^{\prime})}(\Omega_2)\right)^2 $ and $h_4\in (C^{0,\alpha}([L_b,L_2])$. Then the problem \eqref{3-16} has a unique solution $\Psi \in C_{2,\alpha}^{(-1-\alpha;\Gamma_{p,\eta})}(\Omega_1)$ satisfying the estimate
\begin{equation}\label{3-21}
\|\Psi\|_{2,\alpha;\Omega_1}^{(-1-\alpha;\Gamma_{p,\eta})}\leq C_\sharp\bigg(\sum_{j=1}^3\|\mg_j\|
_{1,\alpha;\Omega_1}^{(-\alpha;\Gamma_{p,\eta})}+ \sum_{j=2}^3\|\tilde q_j\|_{1,\alpha;\Omega_2}^{(-\alpha;\Gamma_{\eta}^{\prime})}
 +\|h_4\|_{0,\alpha;[L_b,L_2]}\bigg),
\end{equation}
where the constant $C_\sharp$ depends only on the coefficients $\lambda_i$, $i=1,2, 3$, $ b_5 $ and thus depends only on the background solution.
\end{proposition}
\begin{proof}
Firstly, we show that any $H^1(\Omega_1)$ weak solution to \eqref{3-16} has a
higher regularity $ C_{2,\alpha}^{(-1-\alpha;\Gamma_{p,\eta})}(\Omega_1) $.  To this end,  the first equation in \eqref{3-16} can be  rewritten as  a standard second order elliptic equation for $\Psi$:
\begin{equation*}
\begin{aligned}
&\p_{\eta_1}(\lambda_1(\eta_1)\p_{\eta_1}\Psi)+\lambda_2(\eta_1)
\sum_{j=2}^{3}\p_{\eta_j}^2\Psi
\\
&=\lambda_3(\eta_1)b_5\Psi(L_b,\eta^{\prime})+\p_{\eta_1}\mg_1(\eta_1,|\eta^{\prime}|)+
\sum_{j=2}^{3}\p_{\eta_j}\left(\frac{\eta_j\mg_2(\eta_1,|\eta^{\prime}|)}
{|\eta^{\prime}|}\right)-\frac{\mg_2(\eta_1,|\eta^{\prime}|)}
{|\eta^{\prime}|}+\mg_3(\eta_1,|\eta^{\prime}|).
\end{aligned}
\end{equation*}
Since $\Psi\in H^1(\Omega_1)$,   the trace theorem implies that $ \Psi(L_b,\eta^{\prime})\in L^2(\Omega_2) $. Then one can apply  \cite[Theorems 5.36 and 5.45]{LG13} to get
\begin{equation}\label{3-q}
\begin{aligned}
\|\Psi\|_{C^{0,\alpha}(\overline{\Omega_1})}&\leq C_\sharp\bigg(\|b_5\lambda_3(z_1)\Psi(L_b,\eta^{\prime})\|_{L^2(\Omega_2)}
+\|\mg_1\|_{L^4(\Omega_1)}+\sum_{j=2}^3\bigg\|\frac{\eta_j\mg_2(\eta_1,|\eta^{\prime}|)}
{|\eta^{\prime}|}\bigg\|_{L^4(\Omega_1)}\\
&\quad\quad\ \ +\bigg\|\frac{\mg_2(\eta_1,|\eta^{\prime}|)}
{|\eta^{\prime}|}\bigg\|_{L^2(\Omega_1)}
+\|\mg_3\|_{L^2(\Omega_1)}+\sum_{j=2}^3\|\tilde q_j\|_{L^3(\Omega_2)}+\|h_4\|_{L^3(L_b,L_2)}
\bigg)\\
&\leq C_\sharp\bigg(\|\Psi\|_{H^1(\Omega_1)}+\sum_{i=1}^3\|\mg_i\|
_{1,\alpha;\Omega_1}^{(-\alpha;\Gamma_{p,\eta})}+ \sum_{j=2}^3\|\tilde q_j\|_{1,\alpha;\Omega_2}^{(-\alpha;\Gamma_{\eta}^{\prime})}
 +\|h_4\|_{0,\alpha;[L_b,L_2]}\bigg).
\end{aligned}
\end{equation}
Hence $ b_5\lambda_3(z_1)\Psi(L_b,\eta^{\prime})\in C^{\alpha}(\overline{\Omega_1}) $ and the Schauder estimate in \cite[Theorem 4.6]{LG13}
 would imply that
\begin{equation}\label{3-22}
\|\Psi\|_{2,\alpha;\Omega_1}^{(-1-\alpha;\Gamma_{p,\eta})}\leq C_\sharp\bigg(\|\Psi\|_{H^1(\Omega_1)}+\sum_{i=1}^3\|\mg_i\|
_{1,\alpha;\Omega_1}^{(-\alpha;\Gamma_{p,\eta})}+ \sum_{j=2}^3\|\tilde q_j\|_{1,\alpha;\Omega_2}^{(-\alpha;\Gamma_{\eta}^{\prime})}
 +\|h_4\|_{0,\alpha;[L_b,L_2]}\bigg).
\end{equation}
\par  Next, to show the uniqueness of   the $H^1 (\Omega_1)$ weak solution to \eqref{3-16}, we first investigate the following eigenvalue problem
 \begin{equation}\label{3-23}\begin{cases}
\p_{\eta_2}^2 \beta+ \p_{\eta_3}^2 \beta +\mu \beta=0,\ \ {\rm{in}} \ \Omega_2,\\
(\eta_2\p_{\eta_2} + \eta_3 \p_{\eta_3}) \beta=0,\ \ \ {\rm{on}}\ \ \p \Omega_2.
\end{cases}
\end{equation}
By the standard elliptic theroy in \cite{GT82}, for $\mu<0 $, the problem  \eqref{3-23} is  uniquely solvable.  Note that
\begin{equation*}
\mu \|\beta\|_{L^2(\Omega_2)}^2= \iint_{\Omega_2} \left(\p_{\eta_2} \beta)^2+ (\p_{\eta_3} \beta)^2 \right)\de\eta_2 \de\eta_3.
\end{equation*}
Then there exists a sequence of eigenvalues $0=\mu_1<\mu_2\leq \mu_3 \leq \cdots \leq \mu_i\to\infty$ and the corresponding eigenfunctions $\{\beta_i(\eta_2,\eta_3)\}_{i=1}^{\infty}\in C^{\infty}(\overline{\Omega_2})$ associated with $\mu_i$. The functions $\{\beta_i(\eta_2,\eta_3)\}_{i=1}^{\infty}$ constitute a complete orthonormal basis of $L^2(\Omega_2)$ and are also orthogonal in $H^1(\Omega_2)$.
\par Since $ \Psi\in C^{1,\alpha}(\overline{\Omega_1})\cap C^{2,\alpha}(\Omega_1) $, then $ \Psi$
can be represented as
\begin{equation*}
\Psi(\bm\eta)= \sum_{i=0}^{\infty} X_{i}(\eta_1)\beta_i(\eta_2,\eta_3),
\end{equation*}
where
$X_i(z_1)= \int_{L_b}^{L_2} \Psi(\bm\eta) \beta_i(\eta_2,\eta_3) \de \eta_2 \de \eta_3\in C^{1,\alpha}([L_b,L_2])\cap C^{2,\alpha}((L_b,L_2))$ solves the problem
\begin{equation}\label{3-24}
\begin{cases}
\begin{aligned}
&\lambda_1(\eta_1)X_i^{''}(\eta_1)+\lambda_1^{'}(\eta_1)X_i^{'}(\eta_1)
-\lambda_2(z_1)\mu_iX_i(\eta_1)
-b_5\lambda_3(\eta_1)X_i(L_b)
=0,\\
&X_i^{'}(L_b)-b_5X_i(L_b)=0,\\
&X_i^{'}(L_2)=0.\\
\end{aligned}
\end{cases}
\end{equation}
Suppose that $X_{i} (L_b) = 0$, then the maximum principle and Hopf's lemma show that $X_{i} (\eta_1) \equiv 0 $ for any $\eta_1 \in [L_b, L_2]$. Suppose that $X_{i} (L_b) > 0$, then
\begin{equation}\label{3-25}
\begin{cases}
\begin{aligned}
&\lambda_1(\eta_1)X_i^{''}(\eta_1)+\lambda_1^{'}(\eta_1)
X_i^{'}(\eta_1)-\lambda_2(\eta_1)\mu_iX_i(\eta_1)
=b_5\lambda_3(\eta_1)X_i(L_b)>0
,\\
&X_i^{'}(L_b)=b_5X_i(L_b)>0,\\
&X_i^{'}(L_2)=0.\\
\end{aligned}
\end{cases}
\end{equation}
Assume that there exists $\tilde{\eta}_1 \in [L_b, L_2]$, such that $X_{i} (\tilde{\eta}_1) = \max_{\eta_1 \in [L_b, L_2]} X_{i} (\eta_1) > 0 $. Then the second and the third equations in \eqref{3-25} imply that $\tilde{\eta}_1 \in ( L_b, L_2]$. If $\tilde{\eta}_1 \in (L_b, L_2)$, then $X_{i}^{'} (\tilde{\eta}_1) = 0$, $X_{i}^{''} (\tilde{\eta}_1) \leq 0$, which contradicts to the first equation in \eqref{3-25}. If $\tilde{\eta}_1 = L_2$, then Hopf's lemma yields that $X_{i}' (\eta_2) > 0$, which also contradicts. Similarly, $X_{i} (L_b) < 0$ will induce a contradiction. Hence, $X_{i} (\eta_1) \equiv 0$ for all $\eta_1 \in [L_b, L_2]$. Therefore we get $\Psi \equiv 0$ in $\Omega_1$. We complete the proof of the uniqueness of the $H^1(\Omega_1)$ weak solution to \eqref{3-16}.
\par Next, we can  use Lemma 3.1 and the Fredholm alternatives for elliptic equations and the arguments in \cite[Theorem 8.6]{GT82} to deduce that there exists a unique $H^1 (\Omega_1)$ weak solution to \eqref{3-16}. Furthermore,  the uniqueness helps us to derive the estimate \eqref{3-21}  from \eqref{3-22}.    The invariance of the equation and the boundary datum in \eqref{3-16} under the rotation transform in $(\eta_2,\eta_3)$ plane shows that $ \Psi $ is axisymmetric. This completes the proof of this proposition.
\end{proof}
\par  Proposition 3.2 shows  that  $ \phi_\ast(z_1, z_2) $ is uniquely determined, then $  \Lambda=\phi_\ast(L_b, \me)  $. Hence this proposition actually implies that the following estimate for $w_1$ and $w_2$.
\begin{proposition}
The problem \eqref{3-8} has a unique solution $ (w_1,w_2)\in \left(C_{2,\alpha}^{(-\alpha;\Gamma_{p,z})}(\Omega)\right)^2 $ satisfying  the estimate
\begin{equation}\label{3-26}
\|w_1\|_{2,\alpha;\Omega}^{(-\alpha;\Gamma_{p,z})}
+\|w_2\|_{2,\alpha;\Omega}^{(-\alpha;\Gamma_{p,z})}\leq
C(\delta^2+\sigma)
\end{equation}
and  the compatibility conditions
\begin{equation}\label{3-27}
\p_{z_2} w_1(z_1,0)=(w_2,\p_{z_2}^2 w_2)(z_1,0)=0,  \ \ \forall z_1\in [L_b,L_2].
\end{equation}
\end{proposition}
  \begin{proof}
  It follows from  Proposition   3.2 and the equivalence between $\|\cdot\|_{1,\alpha;\Omega}^{(-\alpha;\Gamma_{p,z})}$ and  $\|\cdot\|_{1,\alpha;\Omega_1}^{(-\alpha;\Gamma_{p,\eta})}$  that the problem \eqref{3-8} has a unique solution $ ( w_1,w_2)\in \left(C_{1,\alpha}^{(-\alpha;\Gamma_{p,z})}(\Omega)\right)^2 $ satisfying
  \begin{equation}\label{3-28}
  \begin{aligned}
  &\|w_1\|_{1,\alpha,\Omega}^{(-\alpha;\Gamma_{p,z})}+
  \|w_2\|_{1,\alpha,\Omega}^{(-\alpha;\Gamma_{p,z})}\\
   &\leq C\bigg(\sum_{j=1}^2\|G_i(\h{\bf w},\h w_5)\|_{1,\alpha;\Omega}^{(1-\alpha;\Gamma_{p,z})}+ \|h_2\|_{1,\alpha;[0,\me)}^{(1-\alpha;\{\me\})}+
   \|h_3\|_{1,\alpha;[0,\me)}^{(-\alpha;\{\me\})}
 +\|h_4\|_{0,\alpha;[L_b,L_2]}\bigg)
 \\
    &\leq C\left(\sigma+
\delta^2\right).\\
   \end{aligned}
  \end{equation}
  Furthermore,  one can further verify  that
 \begin{equation*}
 w_2(z_1,0)=\p_{z_2}w_1(z_1,0)=0, \ \ \forall z_1\in [L_b,L_2].
 \end{equation*}
 \par Next, we estimate $ \|(w_1,w_2)\|_{2,\alpha,\Omega}^{(-\alpha;\Gamma_{p,z})} $. To this end, we rewrite \eqref{3-8} as
  \begin{equation}\label{3-29}
 \begin{cases}
 \begin{aligned}
&\p_{z_1}(\lambda_1(z_1)w_1)
+\lambda_2(z_1)\bigg(\p_{z_2}w_2+\frac{w_2}{z_2}\bigg)=G_3(\h{\bf w},\h w_5),\\
&\p_{z_1}w_2-d_2(z_1)\p_{z_2}w_1
=G_4(\h{\bf w},\h w_5),\\
&w_1(L_b,z_2)=G_5(z_2), \ \
w_1(L_2,z_2)=G_6(z_2)
,\\
&w_2(z_1,0)=0,\ \
w_2(z_1,\me )=h_4(z_1),\\
 \end{aligned}
\end{cases}
\end{equation}
where
\begin{equation*}
\begin{aligned}
G_3(\h{\bf w},\h w_5)&=\frac{\lambda_0(z_1)}{d_1(z_1)}\bigg(G_1(\h{\bf w},\h w_5)-d_4(z_1)\frac{b_3}{b_2}w_1(L_b,z_2)\bigg),\\
 G_4(\h{\bf w},\h w_5)&=G_2(\h{\bf w},\h w_5)-d_5(z_1)\frac{b_3}{b_2} (a_1b_2 w_2(L_b,z_2)+h_2(z_2)),\\
 G_5(z_2)&=a_1b_2\Lambda+b_2R_1(L_b,\me)-\int_{z_2}^{\me}(a_1b_2 w_2(L_b,s)+h_2(s))\de s,\\
G_6(z_2)&=\frac{b_3}{b_2{u}_{b}(L_2)}w_1(L_b,z_2)+h_3(z_2).
 \end{aligned}
 \end{equation*}
Then $ w_1 $ satisfies
   \begin{equation}\label{3-30}
  \begin{cases}
  \begin{aligned}
  &\p_{z_1}\left(\frac{1}{\lambda_2(z_1)}\p_{z_1}(\lambda_1(z_1)w_1)\right)
  +d_2(z_1)
  \left(\p_{z_2}^2w_1+\frac{1}{z_2}\p_{z_2}w_1\right)\\
  &=\p_{z_1}\left(\frac{G_3(\h{\bf w},\h w_5)}{\lambda_2(z_1)}\right)
  +\p_{z_2}G_4(\h{\bf w},\h w_5)+\frac{G_4(\h{\bf w},\h w_5)}{z_2},\\
 & w_1(L_b,z_2)=G_5(z_2),\quad w_1(L_2,z_2)=G_6(z_2),\\ &\p_{z_2}w_1(z_1,0)=0, \qquad w_1(z_1,\me)=w_1(z_1,\me).
  \end{aligned}
  \end{cases}
\end{equation}
Note that $ G_4(z_1,0)=0 $ for $ z_1\in [L_b,L_2] $. By the Schauder estimate in \cite[Theorem 2.23]{LG13}, there holds
\begin{equation}\label{3-31}
\|w_1\|_{2,\alpha;\Omega}^{(-\alpha;\Gamma_{p,z})}\leq C\left(\sum_{i=3}^{4}\|G_i\|_{1,\alpha;\Omega}^{(1-\alpha;\Gamma_{p,z})}
+\sum_{i=5}^{6}\|G_i\|_{1,\alpha;[0,\me)}^{(-\alpha;\{\me\})}
+\|w_1\|_{1,\alpha,\Omega}^{(-\alpha;\Gamma_{p,z})}
\right)\leq C(\sigma+\delta^2).
\end{equation}
This, together with the second equation in \eqref{3-29}, gives
\begin{equation}\label{3-32}
\begin{aligned}
\|\p_{z_1}^2w_2\|_{0,\alpha;\Omega}^{(2-\alpha;\Gamma_{p,z})}
+\|\p_{z_1z_2}^2w_2\|_{0,\alpha;\Omega}^{(2-\alpha;\Gamma_{p,z})}&\leq C(\|w_1\|_{2,\alpha;\Omega}^{(-\alpha;\Gamma_{p,z})}+\|w_2\|_{1,\alpha;\Omega_+}^{(-\alpha;\Gamma_{p,z})}
+\|G_4\|_{1,\alpha;\Omega}^{(1-\alpha;\Gamma_{p,z})}) \\
&\leq C(\sigma+\delta^2).\\
\end{aligned}
\end{equation}
\par Next, for  $ z_1\in [L_b,L_2] $, we derive that $ \p_{z_2}^2w_2(z_1,0)=0 $. It follows from the first equation in \eqref{3-29} and $ w_2(z_1,0)=0 $ that
\begin{equation}\label{3-33}
w_2(z_1,z_2)=\frac{1}{\lambda_2(z_1)z_2} \int_{0}^{z_2}s(\p_{z_1}(\lambda_1(z_1)w_1(z_1,s)-G_3(z_1,s))\de s.
\end{equation}
Then $ w_2(z_1,z_2) $ can be rewritten as
\begin{equation}\label{3-34}
w_2(z_1,z_2)=\frac{1}{z_2}\int_{0}^{z_2}s(R(z_1,s)-R(z_1,0))\de s+\frac{z_2}{2}R(z_1,0),
\end{equation}
where
\begin{equation*}
 R(z_1,z_2)=\frac{1}{\lambda_2(z_1)} \p_{z_1}(\lambda_1(z_1)w_1(z_1,z_2))-G_3(z_1,z_2).
 \end{equation*}
 Thus one gets $ \p_{z_2}R(z_1,0)=0 $ for $ z_1\in [L_b,L_2] $.  Furthermore, it follows from \eqref{3-34} that
 \begin{equation*}
 \p_{z_2}^2w_2=I_1+I_2+I_3,
 \end{equation*}
 where
 \begin{equation*}
 \begin{aligned}
 I_1=\frac{2}{z_2^3}\int_{0}^{z_2}s(R(z_1,s)-R(z_1,0))\de s,\
 I_2=-\frac{1}{z_2}(R(z_1,z_2)-R(z_1,0)),\
 I_3=\p_{z_2}R(z_1,z_2).\\
 \end{aligned}
 \end{equation*}
 Obviously, $ I_3(z_1,0)=0 $ for $ z_1\in [L_b,L_2] $. In addition,
 \begin{equation*}
  \begin{aligned}
 I_2&=-\frac{1}{z_2}(R(z_1,z_2)-R(z_1,0))=-\int_{0}^1\p_{z_2}R(z_1,s z_2)\de s,\\
 I_1&=\frac{2}{z_2^3}\int_{0}^{z_2}s(R(z_1,s)-R(z_1,0))\de s=
 \frac{2}{z_2^3}\int_{0}^{z_2}\left(\int_0^1\p_{z_2}R(z_1,ts)\de t\right)s^2\de s.\\
  \end{aligned}
  \end{equation*}
   Hence for $ z_1\in [L_b,L_2] $,  $ I_1(z_1,0)=I_2(z_1,0)=0$. That is $ \p_{z_2}^2w_2(z_1,0)=0 $. The proof of Proposition  3.3 is completed.
  \end{proof}
  \par  In the following, we are ready to estimate $w_3$,  $w_4$, and $w_5$.
\begin{proposition}
$w_5$, $w_3$,  and $w_4$ are uniquely determined by \eqref{3-3}, \eqref{3-5} and \eqref{3-7}, which satisfy the following estimate
\begin{equation}\label{3-35}
\|w_5\|_{3,\alpha;[0,\me)}^{(-1-\alpha;\{\me\})}
+\|w_3\|_{2,\alpha;\Omega}^{(-\alpha;\Gamma_{p,z})}
+\|w_4\|_{2,\alpha;\Omega}^{(-\alpha;\Gamma_{p,z})}
\leq
C(\delta^2+\sigma)
\end{equation}
and the compatibility conditions
\begin{equation}\label{3-36}
w_5^\prime(0)=w_5^{(3)}(0)=0, \ \ (w_3,\p_{z_2}w_3)(z_1,0)= \p_{z_2}w_4(z_1,0)
 =0, \ \ \forall  z_1\in [L_b,L_2].
\end{equation}
\end{proposition}
\begin{proof}
It follows from \eqref{3-3} that
\begin{equation}\label{3-37}
\begin{aligned}
 \|w_5\|_{2,\alpha;[0,\me)}^{(-\alpha;\{\me\})}&\leq C\left(\|w_1\|_{2,\alpha;\Omega}^{(-\alpha;\Gamma_{p,z})}+ \|R_1(\bm{\Psi}^-(L_b+\h w_5,z_2) - {\bm{\Psi}}_b^-(L_b+\h w_5),\h{\bf  w}(L_b,z_2), \h w_5)\|_{2,\alpha;\Omega}^{(-\alpha;\Gamma_{p,z})}\right)\\
&\leq C\left(
\sigma+\delta^2\right).
\end{aligned}
\end{equation}
Meanwhile, one can follow from \eqref{3-a} to derive that
\begin{equation}\label{3-38}
\begin{aligned}
\|w_5^{\prime}\|_{2,\alpha;[0,\me)}^{(-\alpha;\{\me\})}&\leq C\left(\|w_2\|_{2,\alpha;\Omega}^{(-\alpha;\Gamma_{p,z})}+ \|h_1(\bm{\Psi}^-(L_b+\h w_5,z_2) - {\bm{\Psi}}_b^-(L_b+\h w_5),\h{\bf  w}(L_b,z_2), \h w_5)\|_{2,\alpha;\Omega}^{(-\alpha;\Gamma_{p,z})}\right)\\
&\leq C\left(
\sigma+\delta^2\right).
\end{aligned}
\end{equation}
Furthermore, using \eqref{1-13}, \eqref{3-27} and the explicit expression of $ h_1 $, one can verify that
\begin{equation*}
 w_5^\prime(0)=w_5^{(3)}(0)=0.
 \end{equation*}
 Next, \eqref{3-5} gives that
 \begin{equation}\label{3-40}
(w_3, \p_{z_2}w_3)(z_1,0)=0, \ \ \forall  z_1\in [L_b,L_2] \ \ {\rm{and}}\quad\|w_3\|_{2,\alpha;\Omega}^{(-\alpha;\Gamma_{p,z})}\leq C\delta\sigma .
\end{equation}
Fianlly, it follows from \eqref{3-7}  that the following estimate and compatibility condition hold:
\begin{equation}\label{3-41}
\begin{aligned}
\|w_4\|_{2,\alpha;\Omega}^{(-\alpha;\Gamma_{p,z})}&\leq C\left(\|w_1\|_{2,\alpha;\Omega}^{(-\alpha;\Gamma_{p,z})}+ \|R_4(\bm{\Psi}^-(L_b+\h w_5,z_2) - {\bm{\Psi}}_b^-(L_b+\h w_5),\h{\bf w}(L_b,z_2), \h w_5)\|_{2,\alpha;\Omega}^{(-\alpha;\Gamma_{p,z})}\right)\\
&\leq C\left(\sigma+\delta^2\right),
\end{aligned}
\end{equation}
and
\begin{equation*}
 \p_{z_2}w_4(z_1,0)=0, \ \ \forall  z_1\in [L_b,L_2].
 \end{equation*}
 Combining  \eqref{3-37}-\eqref{3-41} together finishes the proof of the proposition.
\end{proof}
\subsection{Proof of Theorem 2.2}\noindent
\par Now, we start to prove Theorem 2.2. The proof is divided into two steps.
\par {\bf Step 1.  The  boundedness of the operator $\mathcal{T}$.}
\par Given any $ (\hat{\bf{w}},\hat w_5)\in  \mj $, we  define a mapping $ \mt $ as follows
\begin{equation}\label{3-42}
\mt(\hat{\bf{w}},\hat w_5)=({\bf{w}}, w_5),
 \end{equation}
 where  $ ({\bf{w}}, w_5) $ is the solution obtained in Proposition 3.3 and 3.4. Combining \eqref{3-26} and \eqref{3-35}, one derives that
 \begin{equation}\label{3-43}
 \|({\bf{w}}, w_5)\|_{\mj}\leq C_\ast(\sigma+\delta^2).
 \end{equation}
   Setting $\delta=2C_\ast\sigma $ and choosing $ \sigma_0 $ small enough such that  $ 2C_\ast^2\sigma_0\leq\frac{1}{2} $. Then for any $ 0<\sigma<\sigma_0 $, $ C_\ast(\sigma+\delta^2)=\frac{\delta}{2}+2C_\ast^2\sigma\delta<\delta $, hence $ \mt $ maps $ \mj $ into itself.
\par {\bf Step 2. The contraction of the operator $\mathcal{T}$.}
 \par For any two elements $(\hat{{\bf w}}^i, \hat{w}_5^i), i=1,2$, define $({\bf w}^i, w_5^i)=\mathcal{T}(\hat{{\bf w}}^i, \hat{w}_5^i)$ for $i=1,2$. Denote
\begin{equation*}
(\hat{{\bf k}},\h k_5)= (\hat{{\bf w}}^{(1)},\h w_5^{(1)})-(\hat{{\bf w}}^{(2)},\h w_5^{(2)})\quad \text{and}\quad  ({{ \bf k}}, k_5)= ({{\bf w}}^{(1)}, w_5^{(1)})-({{\bf w}}^{(2)},w_5^{(2)}).
\end{equation*}
\par It follows from \eqref{3-8} that $ k_1 $ and $ k_2 $ satisfy
\begin{equation}\label{3-44}
\begin{cases}
\begin{aligned}
&d_1(z_1)\p_{z_1}k_1
+d_2(z_1)\bigg(\p_{z_2}k_2+\frac{k_2}{z_2}\bigg)+ d_3(z_1)k_1+d_4(z_1)\frac{b_3}{b_2}k_1(L_b,z_2)=G_1({{\bf w}}^{(1)}, w_5^{(1)})-G_1({{\bf w}}^{(2)},w_5^{(2)}),\\
&\p_{z_1}k_2-d_2(z_1)\p_{z_2}k_1
+d_5(z_1)\frac{b_3}{b_2}\p_{z_2}k_1(L_b,z_2)=G_2({{\bf w}}^{(1)}, w_5^{(1)})-G_2({{\bf w}}^{(2)},w_5^{(2)}),\\
&\p_{z_2}k_1(L_b,z_2)=a_1b_2 k_2(L_b,z_2)+h_2^{(1)}(z_2)-h_2^{(2)}(z_2),\\
&k_1(L_2,z_2)=
\frac{b_3}{b_2{u}_{b}(L_2)}k_1(L_b,z_2)+h_3^{(1)}(z_2)-h_3^{(2)}(z_2),\\
&k_2(z_1,0)=0,\\
&k_2(z_1,\me )=h_4^{(1)}(z_1)-h_4^{(2)}(z_1),\\
\end{aligned}
\end{cases}
\end{equation}
Then  Proposition 3.3 gives that
\begin{equation}\label{3-45}
\begin{aligned}
 \sum_{j=1}^2\|k_j\|_{2,\alpha;\Omega}^{(-\alpha;\Gamma_{p,z})}
 &\leq
C\bigg(\sum_{j=1}^2\|G_j(\h{\bf w}^{(1)},\h w_5^{(1)})-G_i(\h{\bf w}^{(2)},\h w_5^{(2)})\|_{1,\alpha;\Omega}^{(1-\alpha;\Gamma_{p,z})}
+ \|h_2^{(1)}-h_2^{(2)}\|_{1,\alpha;[0,\me)}^{(1-\alpha;\{\me\})}\\
&\qquad+\|h_3^{(1)}-h_3^{(2)}\|_{1,\alpha;[0,\me)}^{(-\alpha;\{\me\})}
 +\|h_4^{(1)}-h_4^{(2)}\|_{0,\alpha;[L_b,L_2]}\bigg)\\
 &\leq C\sigma
 \|(\hat{{ \bf k}},\h k_5)\|_{\mj}.
 \end{aligned}
\end{equation}
\par Next, it follows \eqref{3-3} and \eqref{3-a} that
\begin{equation}\label{3-46}
\begin{cases}
\begin{aligned}
k_5 (z_2) &= \frac{1}{b_2}k_1(L_b,z_2)- \frac{1}{b_2} (R_1^{(1)}-R_1^{(2)}),\\
k_5^{\prime}(z_2)
&= a_1 k_2(L_b,y_2) +h_1^{(1)}-h_1^{(2)}.
\end{aligned}
\end{cases}
\end{equation}
Thus one can derive
\begin{equation}\label{3-47}
\begin{aligned}
 \|k_5\|_{3,\alpha;[0,\me)}^{(-1-\alpha;\{\me\})}
&\leq C\left(
 \|k_1\|_{2,\alpha;\Omega}^{(-\alpha;\Gamma_{p,z})}
 +\|k_2\|_{2,\alpha;\Omega}^{(-\alpha;\Gamma_{p,z})}
+\|R_1^{(1)}-R_1^{(2)}\|_{2,\alpha;\Omega}^{(-\alpha;\Gamma_{p,z})}
+\|h_1^{(1)}-h_1^{(2)}\|_{2,\alpha;\Omega}^{(-\alpha;\Gamma_{p,z})}\right)\\
&\leq C\sigma
 \|(\hat{{ \bf k}},\h k_5)\|_{\mj}.
 \end{aligned}
\end{equation}
\par Finally, \eqref{3-5}and \eqref{3-7} yield that
\begin{equation}\label{3-48}
\begin{cases}
\begin{aligned}
k_3(z_1,z_2)&=\frac{\hat r^{(1)}(L_b,z_2) u_{\theta}^-(L_b+\h w_5^{(1)}(z_2),z_2)}
{\hat r^{(1)}(z_1,z_2)}-\frac{\hat r^{(2)}(L_b,z_2) u_{\theta}^-(L_b+\h w_5^{(2)}(z_2),z_2)}
{\hat r^{(2)}(z_1,z_2)}\\
k_4(z_1,z_2)&=\frac{b_3}{b_2} k_1(L_b,z_2)+R_4^{(1)}-R_4^{(2)}.
\end{aligned}
\end{cases}
\end{equation}
Hence it holds that
\begin{equation}\label{3-49}
\|k_3\|_{2,\alpha;\Omega}^{(-\alpha;\Gamma_{p,z})}
 +\|k_4\|_{2,\alpha;\Omega}^{(-\alpha;\Gamma_{p,z})}
\leq C\sigma
 \|(\hat{{ \bf k}},\h k_5)\|_{\mj}.
\end{equation}
Combining all the above estimates, one can conclude that
\begin{equation}\label{3-50}
\|({{ \bf k}}, k_5)\|_{\mj}\leq C_\sharp\sigma\|(\hat{{ \bf k}},\h k_5)\|_{\mj}.
\end{equation}
Choosing $ \sigma_0\leq \min\{\frac{1}{4C_\ast^2},\frac{1}{2C_\sharp}\} $, then if $ 0<\sigma<\sigma_0 $,  $ \|({{ \bf k}}, k_5)\|_{\mj}\leq \frac12\|(\hat{{ \bf k}},\h k_5)\|_{\mj} $ so that the mapping $\mathcal{T}$ is a contraction operator in the  norm $\|\cdot\|_{\mj}$.  Thus there exists a unique fixed point $({\bf w},w_5)\in \mathcal{J}$ such that $\mathcal{T}({\bf w},w_5)=({\bf w},w_5)$. It is easy to see that this fixed point is the solution for the problem \eqref{2-31}-\eqref{2-32} and \eqref{2-33} with boundary conditions \eqref{2-29}, \eqref{2-30}, \eqref{2-34} and \eqref{2-35}.
\par
Since the modified Lagrangian transformation is invertible, thus the soultion   transformed back in $(x,r) $-coordinates     satisfies the properties   \eqref{1-17}-\eqref{1-20} in   Theorem 1.3. The proof of
  Theorem 1.3 is completed.
  \par {\bf Acknowledgement.} The research of Zihao Zhang is supported by  the Postdoctoral Fellowship Program of CPSF under Grant Number GZB20250719.
  \par {\bf Data availability.} No data was used for the research described in the article.
    \par {\bf Conflict of interest.} This work does not have any conflicts of interest.

\end{document}